\documentclass[12pt]{article} 
\usepackage[sectionbib]{natbib}
\usepackage{array,epsfig,fancyheadings,rotating}
\usepackage[]{hyperref}  
\usepackage{sectsty, secdot}
\sectionfont{\fontsize{12}{14pt plus.8pt minus .6pt}\selectfont}
\renewcommand{\theequation}{\thesection\arabic{equation}}
\subsectionfont{\fontsize{12}{14pt plus.8pt minus .6pt}\selectfont}
\usepackage[margin=1in]{geometry}

\usepackage{amsmath}
\usepackage{amssymb}
\usepackage{amsfonts}
\usepackage{multirow}
\usepackage{amsthm}

\usepackage[utf8]{inputenc} 
\usepackage[T1]{fontenc}    
\usepackage{hyperref}       
\usepackage{amsfonts}       
\usepackage{nicefrac}       
\usepackage{amsmath}
\usepackage{color}
\usepackage{graphicx}
\usepackage{setspace}
\usepackage[linesnumbered,ruled,vlined]{algorithm2e}
\usepackage{verbatim}
\usepackage{subcaption}
\usepackage{mathtools}
\usepackage{tcolorbox}
\usepackage{amsfonts}
\usepackage{bm}

\usepackage{thmtools}
\usepackage{thm-restate}
\newtheorem{prop}{Proposition}

\newtheorem{definition}{Definition} 
\newtheorem{lemma}{Lemma}
\newtheorem{ex}{Example}
\newtheorem{rem}{rem}

\newcommand{\defeq}{\vcentcolon=}

\def\thx{{\hat\theta_X}}
\def\thy{{\hat\theta_Y}}
\def\thz{{\hat\theta_Z}}

\usepackage{graphicx}
\usepackage{array}
\usepackage{booktabs}
\newcolumntype{N}{>{\centering\arraybackslash}m{.5in}}
\newcolumntype{G}{>{\centering\arraybackslash}m{2in}}

\newcommand{\EE}{E}
\newcommand{\ZZ}{\mathbb{Z}}
\newcommand{\indep}{\rotatebox[origin=c]{90}{$\models$}}

\newcommand{\iid}{\overset{\text{i.i.d.}}{\sim}}
\newcommand{\ep}{\epsilon}

\newcommand{\RR}{\mathbb R}

\newcommand{\twid}{\widetilde}

\newcommand{\la}{\lambda}

\DeclareMathOperator{\var}{Var}

\newcommand{\mscr}[1]{\mathcal #1}

\newcommand{\ul}{\underline}

\begin{document}


\renewcommand{\baselinestretch}{2}

\markright{ \hbox{\footnotesize\rm 
}\hfill\\[-13pt]
\hbox{\footnotesize\rm
}\hfill }

\markboth{\hfill{\footnotesize\rm Jordan Awan and Zhanrui Cai} \hfill}
{\hfill {\footnotesize\rm One Step to Efficient Synthetic Data} \hfill}

\renewcommand{\thefootnote}{}
$\ $\par


\fontsize{12}{14pt plus.8pt minus .6pt}\selectfont \vspace{0.8pc}
\centerline{\large\bf One Step to Efficient Synthetic Data }
\vspace{.4cm} 
\centerline{Jordan Awan$^{*}$\footnote{Corresponding author: Jordan Awan, Department of Statistics, Purdue University, 150 N. University St., West Lafayette, IN 47907, USA. Email: jawan@purdue.edu}, Zhanrui Cai$^{\dagger}$} 
\vspace{.4cm} 
\centerline{\it Purdue University$^{*}$, The University of Hong Kong$^{\dagger}$}
 \vspace{.55cm} \fontsize{9}{11.5pt plus.8pt minus.6pt}\selectfont


\begin{quotation}
\noindent {\it Abstract:}

     A common approach to synthetic data is to sample from a fitted model. We show that under general assumptions, this approach results in a sample with inefficient estimators, and the joint distribution of the sample is inconsistent with the true distribution. Motivated by this, we propose a general method of producing synthetic data that is widely applicable for parametric models,  has asymptotically efficient summary statistics, and is easily implemented and highly computationally efficient. 	 Our approach allows for the construction of both partially synthetic datasets, which preserve certain summary statistics, as well as fully synthetic data, which satisfy differential privacy. In the case of continuous random variables, we prove that our method preserves the efficient estimator with asymptotically negligible error and show through simulations that this property holds for discrete distributions as well. We also provide theoretical and empirical evidence that the distribution from our procedure converges to the true distribution. Besides our focus on synthetic data, our procedure can also be used to perform hypothesis tests in the presence of intractable likelihood functions.

\vspace{9pt}
\par
\end{quotation}\par

\def\thefigure{\arabic{figure}}
\def\thetable{\arabic{table}}

\renewcommand{\theequation}{\thesection.\arabic{equation}}

\fontsize{12}{14pt plus.8pt minus .6pt}\selectfont


\section{Introduction}\label{s:intro}
 
	With advances in modern technology, the government and other research agencies can collect massive amounts of data from individual respondents. These data are valuable for scientific progress and policy research, but they also come with increased privacy risk \citep{lane2014privacy}. Numerous methods of generating \emph{synthetic data} have been proposed to publish useful information while preserving the confidentiality of sensitive information. For a survey, we refer interested readers to \citet[Chapter 3]{hundepool2012statistical}. The goal of synthetic data is to produce a new dataset that preserves the distributional properties of the original dataset, while protecting the privacy of the participating individuals. There are two main types of synthetic data: \emph{partially synthetic data}, which allows for certain statistics or attributes to be released without privacy while protecting the other aspects of the data, and \emph{fully synthetic data}, where all statistics and attributes of the data are protected.
	
	There has also been an increased interest in developing synthetic data with formal privacy guarantees, such as \emph{differential privacy} (DP). 	Differential privacy (DP) was proposed in \citet{dwork2006calibrating} as a framework to develop formally private methods. Methods that satisfy DP require the introduction of additional randomness beyond sampling to obscure the effect of one individual on the output. Intuitively, DP ensures plausible deniability for those participating in the dataset. 

		A common approach to synthetic data is that of	\citet{liew1985data}, which proposes to 
		draw synthetic data from a fitted model, { which we also refer to as the parametric bootstrap.} 
		This forms the basis of the multiple imputation method of synthetic data generation \citep{rubin1993statistical,raghunathan2003multiple,drechsler2011synthetic,jiang2021balancing}. 
	Another approach to synthetic data samples from a \emph{conditional distribution}, preserving certain statistics. The most fundamental perspective of this approach is that of \citet{muralidhar2003theoretical}, who proposes drawing confidential variables from the distribution conditional on the non-confidential variables.

	\subsection{Our contributions}
	{ Related work on} synthetic data largely fits into one of two categories: 1) sampling from a fitted distribution or 2) sampling from a distribution conditional on sample statistics. Our first result, Theorem \ref{thm:TVnaive} shows that in very general settings, the first approach results in a sample with inefficient estimators, and whose distribution is ``inconsistent.'' In particular, we show that the joint distribution of the synthetic sample does not converge in total variation to the true joint distribution as the sample size increases. This result gives a strong indication that {the parametric bootstrap} is not ideal for synthetic data generation. On the other hand, sampling conditional on certain sample statistics can overcome these issues.

	However, there are important limitations to the previous works which sample from a conditional distribution. First, the previous approaches tend to be highly specific to the model at hand and require different techniques for different models. Second, many of the approaches are difficult to implement and computationally expensive, involving complex iterative sampling schemes such as MCMC.

The approach we propose in this paper preserves summary statistics, but unlike previous methods it is applicable to a wide variety of parametric models, easily implemented, and highly computationally efficient. 	Our approach allows for the construction of both partially synthetic datasets, which preserve the summary statistics without formal privacy methods, as well as fully synthetic data which satisfy the strong guarantee of differential privacy (DP). 

{  Our contributions are summarized as follows: }

	 { 
	\begin{itemize}
	    \item We prove that the parametric bootstrap results in inconsistent synthetic data with inefficient estimators.
	    \item We propose a novel method, ``one-step synthetic data,'' which adds one extra step to the parametric bootstrap. Our approach is easily applied as the computations only require efficient estimators for the parameters and the ability to sample the model, and the computational time is proportional to simply fitting the model. 
	    \item We prove that under regularity conditions, our synthetic data procedure preserves an efficient estimator with an asymptotically negligible error. We call this ``efficient synthetic data,'' as its estimators are also efficient.
	    \item We prove that when conditioning on an efficient estimator, the distributions still converge even if the  parameters differ by $O(n^{-1/2})$. We argue that our method is an approximation to a conditional distribution and this gives evidence that the one-step synthetic data asymptotically has the same distribution as the original dataset. 
	    \item We investigate the performance of our procedure in several simulation studies, confirming the theoretical results, and offering numerical evidence that our assumptions can likely be weakened.
	\end{itemize}}
	
	\subsection{Organization}
 The rest of the paper is organized as follows: In Section \ref{s:background}, we review some terminology and set the notation for the paper. In Section \ref{s:bootstrap}, we prove the limitations of synthetic data generated by the parametric bootstrap, showing that the distribution is inconsistent and has inefficient summary statistics. We propose our one-step approach to synthetic data in Section \ref{s:asymptotics} and in the case of continuous random variables, prove that it results in a sample which preserves an efficient estimator with asymptotically negligible error. In Section \ref{s:distribution}, we consider the distribution of the one-step synthetic sample, and prove that a related procedure results in a consistent sample, giving evidence that the one-step procedure itself is consistent. In Section \ref{s:simulations} we perform several simulation studies illustrating 1) the { efficiency of the one-step synthetic data estimators}, even in the case of discrete distributions, 2) the distributional properties of the approach, 3) that the approach can give high quality DP synthetic { data}, and 4) the one-step synthetic data can perform accurate and powerful hypothesis tests on models with intractible likelihood functions. 
	We end in Section \ref{s:conclusions} with some discussion. Proofs and technical details are postponed to Section S3 of the Supplementary Materials.

 \subsection{Related work} 
The approach of sampling from a fitted model is  often used to produce differentially private synthetic data. 
	\citet{hall2013differential} develop DP tools for kernel density estimators, which can be sampled to produce DP synthetic data. 
	\citet{machanavajjhala2008privacy} develop a synthetic data method based on a multinomial model, which satisfies a modified version of DP to accommodate sparse spatial data.	
	\citet{mcclure2012differential} sample from the posterior predictive distribution to produce DP synthetic data, which is asymptotically similar to the \citet{liew1985data}. \citet{ju2022data} provide a general Metropolis-within-Gibbs algorithm that can sample from the posterior predictive distribution, given DP summary statistics, for a wide variety of models and privacy mechanisms. 
	\citet{liu2016model} also use a Bayesian framework: first, they produce DP estimates of the Bayesian sufficient statistics, draw the parameter from the distribution conditional on the DP statistics, and finally sample synthetic data conditional on the sampled parameter.	
	\citet{zhang2017privbayes} propose a method of developing high-dimensional DP synthetic data which draws from a fitted model based on differentially private marginals.  
	{ While not for the purpose of generating synthetic data, \citet{ferrando2022parametric} proposed using the parametric bootstrap to do statistical inference on model parameters, given privatized statistics.}

	\citet{burridge2003information,mateo2004fast,ting2005romm} generate partially synthetic data, preserving the mean and covariance for normally distributed variables. 
	There are also tools, often based on algebraic statistics, to sample conditional distributions preserving certain statistics for contingency tables \citep{karwa2013conditional,chen2006sequential,slavkovic2010synthetic}.
	
	While not focused on the problem of synthetic data, there are other notable works sampling from conditional distributions. A series of works sample from distributions conditional on sufficient statistics focused on the application of hypothesis testing \citep{lindqvist2005monte,lindqvist2007conditional,lindqvist2013exact,lillegard1999exact,engen1997stochastic,lillegard1999exact,taraldsen2018conditional}. \citet{barber2020testing} recently gave a method of sampling conditional on an efficient statistic using the principle of asymptotic sufficiency; they showed that their method results in asymptotically valid $p$-values for certain hypothesis tests.

	In differential privacy, there are also synthetic data methods which preserve sample statistics.  
	\citet{karwa2012differentially} generate DP synthetic networks from the beta exponential random graph model, conditional on the degree sequence. 	
	\citet{li2018privacy} produce DP high dimensional synthetic contingency tables using a modified Gibbs sampler.
	\citet{hardt2012simple} give a distribution-free algorithm to produce a DP synthetic dataset, which approximately preserves several linear statistics.
	
	
	
	While this paper is focused on producing synthetic data for parametric models, there are several non-parametric methods of producing synthetic data, using tools such as regression trees \citep{reiter2005using,drechsler2008accounting}, random forests \citep{caiola2010random}, bagging and support vector machines \citep{drechsler2011empirical}. Recently there has been a success in producing differentially privacy synthetic data using generative adversarial neural networks \citep{xie2018differentially,jordon2018pate,triastcyn2018generating,xu2019ganobfuscator,harder2021dp}. { We also mention a hardness result due to \citet{ullman2020pcps}, which establishes that there is no polynomial time algorithm which can approximately preserve all two-way margins of binary data; our focus on parametric models side-steps this issue.}

	\section{Background and Notation}\label{s:background}
	In this section, we review some background and notation that we use throughout the paper.  
	
	For a parametric random variable, we write $X\sim f_\theta$ to indicate that $X$ has probability density function (pdf) $f_\theta$. 
To indicate that a sequence of random variables from the model $f_\theta$ are independent and identically distributed (i.i.d.), we  write $X_1,\ldots, X_n \iid f_\theta$, and denote $f_\theta^n$ as the joint pdf of $(X_1,\ldots, X_n)$. We use all of the following notations interchangeably: $\ul X =(X_i)_{i=1}^n= (X_1,\ldots, X_n)^{\top}$. We write $\RR^{d\times n}$ to denote the set of all $n$-tuples of elements from $\RR^d$.%

	Our notation for convergence of random variables follows that of \citet{van2000asymptotic}. 
	Let $X$  be a random vector, $X_n$ be a  sequence of random vectors, and $r_n$ be a positive numerical sequence. We write $X_n \overset d \rightarrow X$ to denote that $X_n$ \emph{converges in distribution to} $X$. We write $X_n = o_p(r_n)$	to denote that $X_n/r_n \overset d \rightarrow 0$. We write  $X_n = O_p(r_n)$ to denote that $X_n/r_n$ is \emph{bounded in probability}.
	
	
	For $X\sim f_\theta$, we denote the \emph{score function} as $S(\theta,x) = \nabla_\theta \log f_\theta(x)$, and the \emph{Fisher information} as $I(\theta) = \EE_\theta\left\{S(\theta,X) S^\top(\theta,X)\right\}$, { where $\EE_\theta$ denotes the expectation over the random variable $X$ when $\theta$ is the true parameter}.  An estimator $\hat \theta: \RR^{d\times n} \rightarrow \Theta$ is \emph{efficient} for $\theta$ if for $X_1,\ldots, X_n \iid f_\theta$, we have $\sqrt n \{\hat \theta(\ul X) - \theta\} \overset d \rightarrow N\{0, I^{-1}(\theta)\}$. As shorthand, we will often write $\thx$ in place of $\hat \theta(\ul X)$.

\section{Limitations of the Parametric Bootstrap for Synthetic Data}\label{s:bootstrap}
Sampling from a fitted model, also known as the parametric bootstrap, is a common approach to synthetic data.  
However, the parametric bootstrap has considerable weaknesses when used  to generate synthetic data, in that it results in significantly worse approximations to the true sampling distribution. In this section, we prove that the parametric bootstrap gives inefficient sample statistics and results in ``inconsistent'' synthetic data, where we show that the total variation distance between the true distribution and the parametric bootstrap approximation does not go to zero as $n\rightarrow \infty$. 

The ideal goal of synthetic data is to produce a new dataset $\ul Y$ which is approximately equal in distribution to $\ul X$, where the approximation is measured by total variation distance, $\mathrm{TV}(\ul X,\ul Y)$. At a minimum, we may expect that the distribution of $\ul Y$ approaches the distribution of $\ul X$ as the sample size $n$ grows. We begin with an example that shows that the parametric bootstrap results in suboptimal asymptotics, calling into question the appropriateness of the parametric bootstrap for the generation of synthetic data. 
	
	\begin{ex}\label{s:normal}
		Suppose that $X_1\ldots, X_n \iid N(\mu,1)$. We use the estimator $\hat \mu(\ul X)= n^{-1} \sum_{i=1}^n X_i$ and draw $Z_1,\ldots, Z_n|\hat\mu(\ul X) \iid N\{\hat \mu(\ul X),1\}$. 
		We can compute $\var\{\hat \mu(\ul X)\} = n^{-1} $, whereas $\var\{\hat \mu(\ul Z)\} = 2n^{-1}$. 
		By using the synthetic data $\ul Z$, we have lost half of the effective sample size. We can also derive $(Z_1,\ldots, Z_n)^\top \sim N(\mu \ul 1, I_n + n^{-1} \ul 1\, \ul 1^\top)$, where $\ul 1 = (1,\ldots, 1)^\top$ is a vector of length $n$. Using the formula for the Hellinger distance between Gaussian distributions and Sylvester's Determinant Theorem \citep[p271]{pozrikidis2014introduction}, it is easily calculated that the Hellinger distance between the distributions of $\ul X$ and $\ul Z$ is  $h(\ul X,\ul Z) = [\frac 12 \int\{\sqrt{f_{\ul x}(t)}-\sqrt{f_{\ul z}(t)}\}^{2} \ dt]^{1/2}=\{1- 2^{3/4}(3)^{-1/2}\}^{1/2} \geq .17$, regardless of the sample size $n$,  where $f_{\ul x}$ and $f_{\ul z}$ represent the marginal distributions of the samples $\ul X$ and $\ul Z$ respectively. 
		{ Recall that $h^2(\ul X,\ul Z)\leq \mathrm{TV}(\ul X,\ul Z)$ indicating that the marginal distributions of $\ul X$ and $\ul Z$ do not converge in total variation distance.} 
	\end{ex}
	
	 Example \ref{s:normal} is in fact one instance of a very general phenomenon. In Theorem \ref{thm:TVnaive}, we show that when $\hat \theta(\cdot)$ is an efficient estimator, then $\thz$ is an inefficient estimator for $\theta$, and the distribution of $\ul Z$ is ``inconsistent'' in that the distributions of $\ul Z$ and $\ul X$ do not converge in total variation, as $n\rightarrow \infty$. 

  For the formal statement of Theorem \ref{thm:TVnaive}, we assume that the efficient estimators are locally asymptotically efficient, a property that \citet{beran1997diagnosing} showed is sufficient to ensure that the parametric bootstrap converges correctly. 
  \begin{definition}[Local Asymptotic Equivariance]
    For $\theta\in \Theta\subset \RR^d$, let $X_1,\ldots, X_n\iid f_{\theta}$, and let $\hat\theta(\ul X)$ be an estimator for $\theta$. Call $H_n(\theta)$ be the distribution of $\sqrt n(\hat\theta(\ul X)-\theta)$, which we assume converges  to $H(\theta)$ as $n\rightarrow \infty$. We say that $\hat \theta$ is a \emph{locally asymptotically equivariant (LAE)} estimator at $\theta_0\in \Theta$ if for every convergent sequence $h_n\rightarrow h$, we have $H_n(\theta_0+h_n/\sqrt n)\rightarrow H(\theta_0)$. 
\end{definition}

\begin{restatable}{thm}{thmTVnaive}\label{thm:TVnaive}
Let $X_1,\ldots, X_n\iid f_{\theta_0}$. Suppose that $\hat\theta(\ul X)$ is an efficient estimator, which is LAE at $\theta_0$. Sample $Z_1,\ldots, Z_n |\hat\theta(\ul X) \iid f_{\hat\theta(\ul X)}$. Then
\begin{enumerate}
    \item $\sqrt n \{\hat\theta(\ul Z)-\theta_0\} \overset d \rightarrow N\{0, 2I^{-1}(\theta_0)\}$, whereas $\sqrt n \{\hat\theta(\ul X)-\theta_0\} \overset d \rightarrow N\{0, I^{-1}(\theta_0)\}$, and 
    \item $\mathrm{TV}\{(X_1,\ldots, X_n),(Z_1,\ldots, Z_n)\}$ does not converge to zero as $n\rightarrow \infty$.
\end{enumerate}
\end{restatable}
The use of the total variation distance in Theorem \ref{thm:TVnaive} also has a hypothesis testing interpretation: if the parameter $\theta_0$ were known, then given either the sample $(X_1,\ldots, X_n)$ or $(Z_1,\ldots, Z_n)$, we can always construct a test to discern between the two distributions with power greater than its type I error. So, the samples $\ul X$ and $\ul Z$ never become indistinguishable. In summary, Theorem \ref{thm:TVnaive} shows that the parametric bootstrap is not ideal for the generation of synthetic data. 

\begin{rem}
{ In part, the behavior established in Theorem \ref{thm:TVnaive} is because the synthetic data set is of the same size as the original dataset. We argue that this is an important restriction because if a synthetic dataset is published, users will want to run their own analyses on the synthetic dataset with the assumption that they would get similar results on the original dataset. Modifying the sample size could substantially affect things like confidence interval width and significance in these use cases. For example, if a much larger dataset were generated, then $\thz$ would be closer to $\thx$. However, other aspects of the data would behave differently due to the increased sample size, likely giving artificially narrow confidence intervals. On the other hand, if a much smaller synthetic dataset were generated, its distribution would be closer to the original sampling distribution, but confidence intervals will be too wide, which may prevent users from finding significant features. }

{ 
In the synthetic data literature, the problem illustrated in Theorem \ref{thm:TVnaive} is often addressed by releasing multiple synthetic datasets and using \emph{combining rules} to account for the increased variability due to the synthetic data generation procedure \citep{raghunathan2003multiple,reiter2007multiple,reiter2002satisfying}. However, it still remains that the synthetic data do not follow the same distribution as the original dataset, and the combining rules are often designed for only specific statistics. Furthermore, in the case of differentially private synthetic data, it has been shown that traditional combining rules do not give valid inference, making the problem more complicated \citep{charest2011can}. }
\end{rem}

	\section{One-Step Solution to Synthetic Data}\label{s:asymptotics}
    In this section, we present our synthetic data procedure and show that it has efficient estimators in Theorem \ref{thm:onestep}. We also include a pseudo-code version of our approach in Algorithm \ref{alg:onestep}, to aid implementation.

    While sampling from a fitted model is a common approach to synthetic data, we saw in Theorem \ref{thm:TVnaive} that it results in inferior asymptotics of both the sample estimates as well as the joint distribution of the synthetic data. Our approach avoids the asymptotic problem of Theorem \ref{thm:TVnaive} by producing a sample $(Y_i)_{i=1}^n$ such that $\thy = \thx + o_p(n^{-1/2})$, as proved in Theorem \ref{thm:onestep}. Then marginally, the asymptotic distributions of $\thy$ and $\thx$ are identical. 
    
    { Our method is based around the ability to use the same random ``seed'' at different parameter values. Intuitively, the seed is the source of randomness used to generate the data and is independent of the model parameters. In Example 4, we see that for real-valued continuous data, the seed can sampled from $U(0,1)$, and then transformed into the data using the quantile function. }
    
    The intuition behind our approach is that after fixing the seed, we search for a parameter $\theta^*$ such that when $(Y_i)_{i=1}^n$ are sampled from $f_{\theta^*}$, we have that $\thy =\thx+o_p(n^{-1/2})$.  	To arrive at the value $\theta^*$, we use one step of an approximate Newton method, described in Remark \ref{rem:Newton}. 


    To facilitate the asymptotic analysis, we assume regularity conditions (R0)-(R4).  (R0) ensures that the seed has a known distribution, and that we have a method of transforming the seed into the data (see Example \ref{ex:locationScale} for an example). The assumption that $\Omega$ is bounded is very mild, as we can always use a change of variables to make the seed have bounded support. (R1)-(R3) are similar to standard conditions to ensure that there exists an efficient estimator, which are relatively mild and widely assumed in the literature \citep{serfling2009approximation,lehmann2004elements}.  Assumption (R4) is likely much stronger than needed, but ensures that several quantities, including the transformation $X_\theta$, vary smoothly in their parameters; this assumption is important to allow for the interchange of several derivatives in the proof of Theorem \ref{thm:onestep}. { Since (R4) requires that the density is continuous in $x$, this assumption also limits Theorem \ref{thm:onestep} to continuous distributions.}

     In this section, we will prove that our procedure satisfies $\thy=\thx+o_p(n^{-1/2})$ for continuous random variables which satisfy the regularity conditions (R0)-(R4). 
    
	\begin{itemize}
	\item   [(R0)]  Let $(\Omega, \mscr F, P)$ be a probability space of the \emph{seed} $\omega$,  where $\Omega\subset \RR^m$ { is a bounded sample space, $\mscr F$ is a $\sigma$-algebra on $\Omega$, and $P$ is a probability measure with a continuous density $\pi$}. Let $X_\theta:\Omega \rightarrow \RR^d$ be a measurable function, 
 where 	$\theta$ lies in a compact space $\Theta \subset \RR^p$. Let  $f_\theta:\RR^d \rightarrow \RR^{\geq 0}$ denote the density of the random variable $X_\theta(\omega)$.
	
		\item [(R1)]  Let $\theta_0\in \Theta\subset \RR^p$ be the true parameter. Assume there exists an open ball $B(\theta_0)\subset \Theta$ about $\theta_0$, the model $f_\theta$ is identifiable, and that the set $\{x\in \RR^d\mid f_\theta(x)>0\}$ does not depend on $\theta$.
		\item [(R2)] The pdf $f_\theta(x)$ has three derivatives in $\theta$ for all $x$ and 
		there exist functions $g_i(x)$, $g_{ij}(x)$, $g_{ijk}(x)$ for $i,j,k=1,\ldots, p$  such that for all $x$ and all $\theta \in B(\theta_0)$,
		\[\left| \frac{\partial f_{\theta}(x)}{\partial \theta_i}\right|\leq g_i(x),\qquad
		\left| \frac{\partial^2 f_{\theta}(x)}{\partial \theta_i\partial \theta_j}\right|\leq g_{ij}(x),\qquad
		\left| \frac{\partial^3 f_{\theta}(x)}{\partial \theta_i\partial \theta_j\partial \theta_k}\right|\leq g_{ijk}(x).\]
		We assume that each $g$ satisfies $\int g(x) \ dx<\infty$ and $\EE_\theta g_{ijk}(X)<\infty$ for $\theta \in B(\theta_0)$. { Furthermore, we assume that there exists functions $h_{ij}(x)$ for $i,j=1,2,\ldots, p$ such that $|S(\theta,x)S^\top (\theta,x)|_{ij}\leq h_{ij}(x)$ for all $x$ and $\theta\in B(\theta_0)$, and that $\EE_\theta h_{ij}(X)<\infty$ for all $\theta\in B(\theta_0)$.}
		\item [(R3)] The Fisher information matrix $I(\theta_0) = \EE_{\theta_0} \{S(\theta_0,X)S^\top(\theta_0,X)\}$ consists of finite entries, and is positive definite.
		\item [(R4)]  The quantities $X_\theta(\omega)$, $\frac \partial{\partial\theta_i} X_\theta(\omega)$, $\frac{\partial}{\partial \theta_i}\log f_\theta(x)$, $\frac{\partial}{\partial \theta_i\partial x_k} \log f_\theta(x)$, $f_\theta(x)$, and $\frac{\partial}{\partial\theta_i} f_\theta(x)$ all exist and are all continuous in $\theta$, $x$, and $\omega$.
	\end{itemize}
Assumption (R0) tells us that to produce a sample $X\sim f_\theta$, we can first sample the seed $\omega \sim P$ and then transform the seed into $Y \defeq X_\theta(\omega)$. This procedure results in a sample equal in distribution: $X\overset d = Y$.  By (R4), $X_\theta$ is assumed to be continuously differentiable in $\theta$, so we have that the mapping $X_\theta(\cdot)$ is smooth as $\theta$ varies. 

An idealized version of our algorithm works as follows: given $\thx$ computed from the original dataset, we first sample the seeds $\omega_1,\ldots, \omega_n\iid P$, and then while holding $\omega_1,\ldots, \omega_n$ fixed, solve for the value $\theta^*$ which satisfies:

\begin{equation}\label{eq:thetaStar}
\hat\theta\{X_{\theta^*}(\ul \omega)\} = \thx,
\end{equation}
thereby ensuring that the new sample has the same value of $\hat\theta$ as the original sample $\ul X$. Finally, we can produce our synthetic data $Y_i = X_{\theta^*}(\omega_i)$. This idealized version has been previously employed when the statistic $\thx$ is a \emph{sufficient statistic}, and used mainly for the task of hypothesis testing \citep{lindqvist2005monte,lindqvist2007conditional,lindqvist2013exact,lillegard1999exact,engen1997stochastic,lillegard1999exact,taraldsen2018conditional}. There are many settings where an exact solution to \eqref{eq:thetaStar} exists, such as location-scale families (Example \ref{ex:locationScale}). { However, in general it may be challenging to find an exact solution to Equation \eqref{eq:thetaStar}, and a solution may not even exist.}



\begin{ex}\label{ex:locationScale}
 In the case of continuous real-valued random variables, we can be more explicit about the ``seeds.''
 Recall that for $U\sim U(0,1)$, $F^{-1}_{\theta}(U) \sim f_{\theta}$ where $F^{-1}_{\theta}(\cdot)$ is the quantile function. So in this case, the  distribution $P$ can be taken as $U(0,1)$, and $X_{\theta}(\cdot)$ can be replaced with $F^{-1}_{\theta}(\cdot)$. 
 
 If $f_\theta$ is a location-scale family, where $\theta =(m,s)$, then there exists an explicit solution to \eqref{eq:thetaStar}. Just as above, we set $X_\theta(\cdot) = F_{m,s}^{-1}(u_i)=sF^{-1}_{0,1}(\omega_i)+m$, where $u_i \iid U(0,1)$, and we used the location-scale formula for the quantile function. Suppose that $\hat m$ and $\hat s$ are estimators of $m$ and $s$ such that $\hat m(ax+b)=a\hat m(x)+b$ and $\hat s(ax+b)=a\hat s(x)$. In the case of the normal distribution, $\hat m$ is the sample mean, and $\hat s$ is the sample standard deviation. 
 
 Then, for $\omega_1,\ldots, \omega_n\iid U(0,1)$, call $Z_i = F^{-1}_{0,1}(\omega_i)$. Then
 \[Y_i = \{Z_i-\hat m(\ul Z)\}\frac{\hat s(\ul X)}{\hat s(\ul Z)}+\hat m(\ul X),\]
 satisfies $\hat m(Y)=\hat m(X)$ and $\hat s(Y)=\hat s(X)$. We see that $Y_i = X_{\theta^*}(\omega_i)$, where $\theta^*=(m^*,s^*)$, where $m^*=\hat m(\ul X)-\hat m(\ul Z) \frac{\hat s(\ul X)}{\hat s(\ul Z)}$ and $s^* = \frac{\hat s(\ul X)}{\hat s(\ul Z)}$. 
\end{ex}

To avoid solving Equation \eqref{eq:thetaStar} exactly, in Theorem \ref{thm:onestep} and Algorithm \ref{alg:onestep}, we propose { an approximate solution which can be viewed as one step of an approximate Newton method, as described in Remark \ref{rem:Newton}.}

    { \begin{rem}[One-step approximate Newton method]\label{rem:Newton}
    The ``one-step'' plugin value $\theta^*$ proposed in Theorem \ref{thm:onestep} can be viewed as one step of an approximate Newton method, which tries to solve $0=\frac 1n\sum_{i=1}^n S\{\thx,X_\theta(\omega_i)\}$. If there is a unique solution to this score equation, then this can be viewed as an indirect way of formulating Equation \eqref{eq:thetaStar}. The approximate Newton method would update 
    $\theta_{n+1}=\theta_n-I^{-1}(\theta_n) n^{-1} \sum_{i=1}^n S\{\thx,X_{\theta_n}(\omega_i)\},$
    where in Lemma 16 of the Supplementary Material, it is shown that $I(\theta)$ is the expected (matrix) derivative of $S\{\thx,X_\theta(\omega)\}$ with respect to $\theta$. Using $\thx$ as an initial value, the one-step approximate solution is 
    $\theta^*_{(1)}=\thx-I^{-1}(\thx)n^{-1} \sum_{i=1}^n S\{\thx,X_{\thx}(\omega_i)\}.$
    Then, using a Taylor expansion for $\thz$, where $Z_i = X_{\thx}(\omega_i)$, we see that 
    $\thz-\thx=I^{-1}(\thx)n^{-1} \sum_{i=1}^n S\{\thx,X_{\thx}(\omega_i)\}+o_p(n^{-1/2}).$
    Substitution gives $\theta^*_{(1)}=\thx-(\thz-\thx)+o_p(n^{-1/2})=2\thx-\thz+o_p(n^{-1/2})$, 
    which motivates our choice of $\theta^*=\mathrm{Proj}_{\Theta}(2\thx-\thz)$ as the plugin value used in Theorem \ref{thm:onestep} (the projection is only needed if $2\thx-\thz$ lies outside of the parameter space $\Theta$). Finally, note that $\theta^*=\theta_0+O_p(n^{-1/2})$, since Theorem \ref{thm:TVnaive} established that both $\thx$ and $\thz$ are $\sqrt n$-consistent estimators of $\theta_0$. 
    \end{rem}
}

The following Theorem shows that regardless of whether a solution to Equation \eqref{eq:thetaStar} exists, the one-step procedure preserves the efficient statistic up to $o_p(n^{-1/2})$. 

\begin{restatable}{thm}{thmonestep}\label{thm:onestep}
Assume that (R0)-(R4) hold. Let $X_1,\ldots, X_n \iid f_{\theta_0}$ and let $\omega_1,\ldots, \omega_n\iid P$. 
Set 
$\theta^* =\mathrm{Proj}_{\Theta}\left(2 \thx-\thz\right)$, where  $\hat \theta$ is an efficient estimator, $(Z_i)_{i=1}^n = (X_{\thx}(\omega_i))_{i=1}^n$, { and $\mathrm{Proj}_{\Theta}(x)$ maps $x$ to the nearest point in $\Theta$ in terms of Euclidean distance}.
 Then for $(Y_i)_{i=1}^n = (X_{\theta^*}(\omega_i))_{i=1}^n$, we have  $\thy =  \thx + o_p(n^{-1/2})$.
\end{restatable}

Theorem \ref{thm:onestep} shows that our one-step approach to synthetic data outperforms the parametric bootstrap in terms of the first result in Theorem \ref{thm:TVnaive}: whereas sampling from the fitted model results in estimators with inflated variance, the one-step approach gives a sample whose estimator $\thy$ is equal to $\thx$ up to an asymptotically negligible error of $o_p(n^{-1/2})$.

The restriction of Theorem \ref{thm:onestep} to continuous distributions, due to (R4), 
 cannot be weakened with our current proof technique (which relies on a derivative with respect to $x$). However, we  offer numerical evidence through simulations that the result of Theorem \ref{thm:onestep}  seems to hold for discrete distributions as well. This suggests that  it may be possible to weaken the assumptions of Theorem \ref{thm:onestep}.

\begin{rem}[Seeds]
When implementing the procedure of Theorem \ref{thm:onestep}, it may be convenient to use numerical seeds. For example in R, the command \texttt{set.seed} can be used to emulate the result of drawing $Z_i$ and $Y_i$ with the same seed $\omega_i$. In Algorithm \ref{alg:onestep}, we describe the one-step procedure in pseudo-code. One must be careful with this implementation to ensure that \texttt{rsample} varies smoothly in $\theta$. 
\end{rem}

    \begin{algorithm}[t]
    \caption{One-Step Synthetic Data Pseudo-Code in R}
    \label{alg:onestep}
        	\setstretch{1}
    	\small
\SetKwInput{KwInput}{Input}                
\DontPrintSemicolon
  
  \KwInput{ Seed $\omega$, efficient estimator $\thx$, function \texttt{theta\_hat}(\ul y) to compute $\hat\theta(\ul y)$, function \texttt{rsample}($\theta$) which samples $n$ i.i.d. samples from $f_\theta$ using the seed $\omega$.\;}
  \texttt{set.seed}($\omega$)\;
 $\ul Z$ = \texttt{rsample}($\thx$)\;
 $\thz$ = \texttt{theta\_hat}($\ul Z$)\;
  $\theta^* = 2\thx-\thz$\;
  \If{$\theta^* \not\in \Theta$}
  {$\theta^*= \mathrm{Proj}_{\Theta}(\theta^*)$\;}
  \texttt{set.seed}($\omega$)\;
  $\ul Y$ = \texttt{rsample}($\theta^*$)\;
  \SetKwInput{KwOutput}{Output}              
    \KwOutput{$Y_1,\ldots, Y_n$}
\end{algorithm}

	
	\begin{rem}[DP fully synthetic data]
	The one-step procedure results in a sample $\ul Y$, which is conditionally independent of $\ul X$, given $\thx$. This is called \emph{partially synthetic data} because all aspects of $\ul X$ are protected except for $\thx$. Partially synthetic data can be appropriate in some settings, but in others we may require $\ul Y$ to satisfy a stronger privacy guarantee such as differential privacy. 
	We can easily use the one-step procedure to obtain a DP fully synthetic sample by using a DP efficient estimator as $\thx$. \citet{smith2011privacy} showed that under conditions similar to (R1)-(R3),  there always exists DP efficient estimators by using the subsample and aggregate technique. By the post-processing property of DP, the sample $\ul Y$ then has the same DP guarantee that $\thx$ has \citep[Proposition 2.1]{dwork2014algorithmic}. Theorem \ref{thm:onestep} is still valid when using a DP estimator for $\thx$ as only the efficiency of $\thx$ is used in the proof. In fact, the estimator $\hat\theta$ applied to the intermediate sample $\ul Z$ need not be the same one as used for $\thx$. For improved finite sample performance, it may be beneficial to use a non-private estimator for $\thz$ even when $\thx$ satisfies DP. See Section \ref{s:beta} where we investigate the performance generating of DP synthetic data using Algorithm \ref{alg:onestep}.
	\end{rem}
	
{	\begin{rem}
	    As indicated in Remark \ref{rem:Newton}, our plugin value $\theta^*$ can be viewed as one step of an approximate Newton method. Our method can also be viewed as one step of the iterative bootstrap, with a batch size of 1 \citep{guerrier2019simulation}, which is traditionally used to de-bias an inconsistent estimator. Both of these perspectives indicate that iterating the procedure could potentially reduce the error between $\thy$ and $\thx$ even further, and in some circumstances even find an exact solution. We leave it to future work to study the benefits of this iterative version, as well as conditions for convergence.
	\end{rem}}

\section{Investigating the Distribution of the One-Step Samples}\label{s:distribution}
In Section \ref{s:asymptotics}, we showed that the one-step approach to synthetic data solves one issue of the parametric bootstrap,  by preserving efficient estimators. The second problem of the parametric bootstrap was that the synthetic samples did not converge in total variation distance to the true joint distribution. 
%
In this section, we give evidence that the distribution of the one-step samples approximates the true joint distribution. 

Consider the one-step procedure: Let $\thx$ be given and draw $\omega_1,\ldots,\omega_n\iid P$. Set $Z_i = X_{\thx}(\omega_i)$ and call $\theta^* = 2\thx-\thz$. Finally, set $Y_i = X_{\theta^*}(\omega_i)$, which we know satisfies $\hat\vartheta \defeq \thy =\thx+o_p(n^{-1/2})$ (under assumptions (R0)-(R4)). Now, suppose that we knew the values $\theta^*$ and $\hat\vartheta$ beforehand and conditioned on them. The following Lemma shows that $Y_1,\ldots, Y_n\{\theta^*,\hat\theta(\ul Y)=\hat\vartheta\} \approx f_{\theta^*}^n\{y_1,\ldots, y_n\mid \hat\theta(\ul y)=\hat\vartheta\}$. The approximation comes from modifying the procedure slightly. Given both $\theta^*,\hat\vartheta\in \Theta$, we produce $\ul Y^{\theta^*}$ as follows: sample $\omega_i\iid P$, and set $Y_i^{\theta^*} = X_{\theta^*}(\omega_i)$. For this modified procedure, we have exactly $Y_1^{\theta^*},\ldots, Y_n^{\theta^*}\large|\{\hat\theta(\ul Y)=\hat\vartheta\} \sim f_{\theta^*}^n\{y_1,\ldots, y_n\mid \hat\theta(\ul y)=\hat\vartheta\}$, which we prove in Lemma \ref{lem:conditional}. { The key difference between this modified procedure and the one in Theorem \ref{thm:onestep} is that here we start with $\theta^*$ and condition on $\hat\theta(\ul Y)=\hat\vartheta$, whereas in the original procedure $\theta^*$ is a function of $\ul\omega$.}


\begin{restatable}{lemma}{lemconditional}\label{lem:conditional}
We assume (R0) and use the notation therein. Let $\hat\vartheta,\theta^*\in \Theta$ such that there exists  $\ul \omega$ which solves the equation  $\hat\theta\{X_{\theta^*}(\ul\omega)\}=\hat\vartheta$. Let $\omega_1,\ldots, \omega_n \iid P$ and call $Y_{i}^{\theta^*} = X_{\theta^*}(\omega_i)$. Then $Y_1^{\theta^*},\ldots, Y_n^{\theta^*} |\{\hat\theta(\ul Y^{\theta^*})=\hat\vartheta\} \sim f_{\theta^*}^n\{y_1,\ldots, y_n \mid \hat\theta(\ul y) = \hat\vartheta\}$.
\end{restatable}

Lemma \ref{lem:conditional} suggests that the conditional distribution of $\ul Y$ is related to the conditional distribution of $\ul X$. Theorem \ref{thm:klGeneral} shows that when this is the case, the marginal distributions are also closely related.  Theorem \ref{thm:klGeneral} is similar in spirit to a result of \citet{le1956asymptotic} which showed that efficient estimators are \emph{asymptotically sufficient}, meaning that with large $n$ an approximate equivalent likelihood function can be constructed which only involves the parameter and the efficient estimator. First, we need an additional assumption on the distribution of $\hat\theta(\ul X)$. 


\begin{itemize}
 \item [(R5)] Let $\thx$ be a randomized efficient estimator of $\theta$, with conditional density $g_n(\thx|\ul x)$. We assume that there exists a sequence $(M_n)_{n=1}^\infty$ such that $g_n(\thx|\ul x)\leq M_n$ for all values of $\thx$ and $\ul x$. Let $g_{\theta,n}(\cdot)$ be the marginal density of $\hat\theta(\ul X)$ based on the sample $X_1,\ldots, X_n \iid f_\theta$. We assume that there exists functions $G_{ijk}(\hat\vartheta)$ such that  $\left| \frac{\partial^3}{\partial \theta_i\partial \theta_j\partial \theta_k} \log g_{\theta,n}(\hat\theta)\right|\leq nG_{ijk}(\hat\vartheta)$ for all $\hat\vartheta$, all $n\geq 1$,  and all $\theta \in B(\theta_0)$, where $\EE_{\hat\vartheta \sim \theta}G_{ijk}(\hat\vartheta)<\infty$.
\end{itemize}

 In (R5), we consider that $\thx$ is a randomized statistic, so we can also write $g_n(\hat\vartheta\mid \ul X)$ to represent the distribution of $\thx$ given $\ul X$ (which we assume does not depend on $\theta$, since  $\thx$ is a statistic). Any deterministic statistic can be expressed as a limit of randomized statistics, where the noise due to randomness goes to zero. For example, \citet{barber2020testing} consider statistics which are zeros of the  noisy score function, where the noise is normally distributed and $o_p(n^{-1/2})$.

{ Theorem \ref{thm:klGeneral} shows that if we condition on the same efficient estimator, then considering the sample $\ul X$ generated from the true parameter $\theta$ or $\ul Y$ generated from a sequence $\theta_n=\theta+O(n^{-1/2})$, the marginal distributions of $\ul X$ and $\ul Y$ converge in KL divergence. By Pinsker's inequality, we know that $\mathrm{TV}(\ul X,\ul Y)\leq \sqrt{\frac 12 \mathrm{KL}(\ul X,\ul Y)}$, establishing convergence in total variation distance as well. Note that (R4) is not needed for Theorem \ref{thm:klGeneral}. As such, this theorem applies to both continuous and discrete distributions. }

\begin{restatable}{thm}{thmklGeneral}\label{thm:klGeneral}
Under assumptions  (R0)-(R3), let $\theta\in \Theta$ and let $\theta_n$ be a sequence of values in $\Theta$. Let $\hat\theta(\cdot)$ be a randomized estimator based on a sample $X_1,\ldots, X_n \iid f_\theta$, with conditional distribution $g_n(\hat\vartheta\mid \ul X)$ and marginal distribution $g_{\theta,n}(\hat\theta)$, which satisfy (R5). Then the KL divergence between the marginal distributions of $X_1,\ldots, X_n$ and $Y_1,\ldots,Y_n \sim f^n_{\theta_n}\{y_1,\ldots, y_n \mid \hat\theta(\ul y)=\hat\theta(\ul X)\}$, where $\hat\theta(\ul X)\sim g_{\theta,n}(\hat\vartheta)$ is 
\begin{align*}
\mathrm{KL}\left(X_1,\ldots,X_n\middle|\middle| Y_1,\ldots, Y_n\right)
=o(n) \lVert \theta_n-\theta\rVert^2 +O(n) \lVert \theta_n-\theta\rVert^3.\end{align*}
In particular, if $\theta_n-\theta=O(n^{-1/2})$, 
then the above KL divergence goes to zero as $n\rightarrow \infty$.
\end{restatable}

Noting that $\theta^* =\theta+O_p(n^{-1/2})$ and combining Lemma \ref{lem:conditional} with Theorem \ref{thm:klGeneral} along with the discussion at the beginning of this section suggests that the distribution of the one-step synthetic data approaches the distribution of the original $\ul X$ as the sample size increases giving consistent synthetic data. This suggests that the one-step synthetic data avoids the problem of property 2 of Theorem \ref{thm:TVnaive}, which showed that the parametric bootstrap resulted in inconsistent synthetic data.  

In Section \ref{s:burr} we show that in the case of the Burr distribution, the one-step synthetic data is indistinguishable from the true distribution in terms of the Kolmogorov-Smirnov (K-S) test, whereas for the parametric bootstrap the K-S test has significantly higher power, indicating that the one-step synthetic data is asymptotically consistent, whereas the parametric bootstrap synthetic data is inconsistent. 
		\section{Examples and Simulations}\label{s:simulations}
		In this section, we demonstrate the performance of the  one-step synthetic data in several simulations. In Sections \ref{s:burr} and \ref{s:log-linear}, we produce synthetic data from the Burr distribution as well as a log-linear model. In Section \ref{s:beta}, we produce differentially private synthetic data for the beta distribution. In Section \ref{s:DPtesting} we use our methods to derive a hypothesis test for the difference of proportions under differential privacy. 
		\subsection{Burr type XII distribution}\label{s:burr}


		The Burr Type XII distribution, denoted $\mathrm{Burr}(c,k)$, also known as the Singh–Maddala distribution, is a useful model for income \citep{mcdonald2008some}. The distribution has pdf $f(x)=ckx^{c-1}(1+x^c)^{-(k+1)}$, with support $x>0$. Both $c$ and $k$ are positive. The Burr distribution was chosen for our first simulation because 1)  the data are one-dimensional, allowing for the Kolmogorov-Smirnov (K-S) test to be applied, and 2) as it is not exponential family or location-scale, deriving the exact conditional distribution,  given the MLE $\hat\theta$, is non-trivial.
		
		 First, we will use this example to illustrate the notation of our theory. Suppose we are given observations $X_1,\ldots, X_n\iid \mathrm{Bur}(c,k)$, from unknown values of $c$ and $k$. Let $\hat \theta$ be the MLE (not available in closed-form). The procedure works as follows: Let $\omega_i\iid U(0,1)$ for $i=1,\ldots, n$. We define the function $X_\theta(\omega)= F^{-1}_\theta(\omega)$, where $F^{-1}_\theta(\omega)=\{(1-\omega)^{-1/k}-1\}^{1/c}$. Then, we first set $Z_i = F^{-1}_{\thx}(\omega_i)$, and after computing $\thz$, our synthetic data is $Y_i = F^{-1}_{\theta^*}(\omega_i)$, where $\theta^* = \max(2\thx-\thz,0)$, with the max  applied to both entries of $\theta$.

		 In the following, we conduct a simulation study to verify that the samples generated using Algorithm \ref{alg:onestep} are indistinguishable from the original unknown distribution, as tested via the K-S test. 
		
			For the simulation, we set $c=2$ and $k=4$, and denote $\theta = (c,k)$. Let $\hat \theta_{MLE}$ be the maximum likelihood estimator (MLE). 
			We draw $X_i \iid \mathrm{Burr}(2,4)$, $Z_i \iid \mathrm{Burr}\{\hat \theta_{MLE}(\ul X)\}$, and  $(Y_i)_{i=1}^n$ from Algorithm \ref{alg:onestep}. 
			The simulation is conducted for $n\in \{100,1000,10000\}$ with results averaged over 10000 replicates for each $n$.
			

        \begin{figure}
        \begin{minipage}[c]{\linewidth}
        \captionof{table}{Empirical power of the Kolmogorov-Smirnov test for $\mathrm{Burr}(2,4)$ at type I error $.05$. $(X_i)$ are drawn i.i.d from $\mathrm{Burr}(2,4)$, $(Z_i)$ are drawn i.i.d from $\mathrm{Burr}\{\thx\}$, and $(Y_i)$ are from Algorithm \ref{alg:onestep}. Results are averaged over 10000 replicates, for each $n$. Standard errors are approximately $0.0022$ for lines 1 and 3, and  $0.0036$ for line 2.}
        		\label{tab:burrKS}
        		\centering
        		\begin{tabular}{llll}
        	\phantom{x}	$n$:	& 100 & 1000 & 10000 \\ 
        			\hline
        			$(X_i)$  & 0.0471 & 0.0464 & 0.0503 \\ 
        			$(Z_i)$  & 0.1524 & 0.1541 & 0.1493 \\ 
        			$(Y_i)$  & 0.0544 & 0.0489 & 0.0485 \\ 
        			\hline
        		\end{tabular}
        		\end{minipage}
        	\end{figure}

			We calculate the empirical power of the K-S test, comparing each sample with the true distribution $\mathrm{Burr}(2,4)$, at type I error $.05$. The results are presented in Table \ref{tab:burrKS}. We see that the $(X_i)$ have empirical power approximately $.05$, confirming that the type I error is appropriately calibrated. We also see that the K-S test using  $(Y_i)$ has power approximately $.05$, indicating that the empirical distribution of the one-step samples $(Y_i)$ is very close to the true distribution. On the other hand, we see that the K-S test with $(Z_i)$ has power $.15$, significantly higher than the type I error, indicating that the parametric bootstrap samples $(Z_i)$ are from a fundamentally different distribution than the $(X_i)$. This result is in agreement with Theorem \ref{thm:TVnaive} and the results of Section \ref{s:distribution}.  


\subsection{Log-linear model for seatbelt data}\label{s:log-linear}
		
	\begin{figure}
	\begin{minipage}[c]{\linewidth}
				\captionof{table}{Recorded injuries according to seatbelt use, gender, and location. Source: \citet[Table 8.8]{agresti2003categorical}. Originally credited to Cristanna Cook, Medical Care Development, Augusta, Maine.}
		\label{tab:seatbelt}
		\centering
		\begin{tabular}{lllrr}
			\multicolumn{3}{c}{\textbf{}} &   \multicolumn{2}{c }{\textbf{Injury}} \\  
			\cmidrule(ll){4-5}
			\textbf{Gender}&\textbf{Location}&\textbf{Seatbelt}&\multicolumn{1}{c}{No}&\multicolumn{1}{c}{Yes}\\
			\midrule
			Female&Urban&No&7,287&996\\ 
			&&     Yes&11,587&759\\
			&Rural&No&3,246&973\\
			&&     Yes&6,134&757\\
			Male&Urban&No&10,381&812\\
			&&Yes&10,969&380\\
			&Rural&No&6,123&1,084\\
			&&Yes&6,693&513\\
			\bottomrule
		\end{tabular}
\end{minipage}
	\end{figure}
		This example is based on a dataset of of 68,694 passengers in automobiles and light trucks involved in accidents in the state of Maine in 1991. Table \ref{tab:seatbelt} tabulates passengers according to gender (G), location (L), seatbelt status (S), and injury status (I).  This example { gives numerical evidence} that the result of Theorem \ref{thm:onestep} holds even for discrete distributions.

		As in \citet{agresti2003categorical}, we fit a hierarchical log-linear model based on all one-way effects and two-way interactions. The model is summarized in Equation \eqref{eq:loglinear}, where $\mu_{ijk\ell}$ represents the expected count in bin $i,j,k,\ell$.  The  parameter $\la_i^G$ represents the effect of gender, and parameter $\la_{ij}^{GL}$ represents the interaction between gender and location. The other main effects and interactions are analogous.

				\begin{figure}[ht]
		  \begin{subfigure}[t]{0.48\textwidth}
		\includegraphics[width=\textwidth]{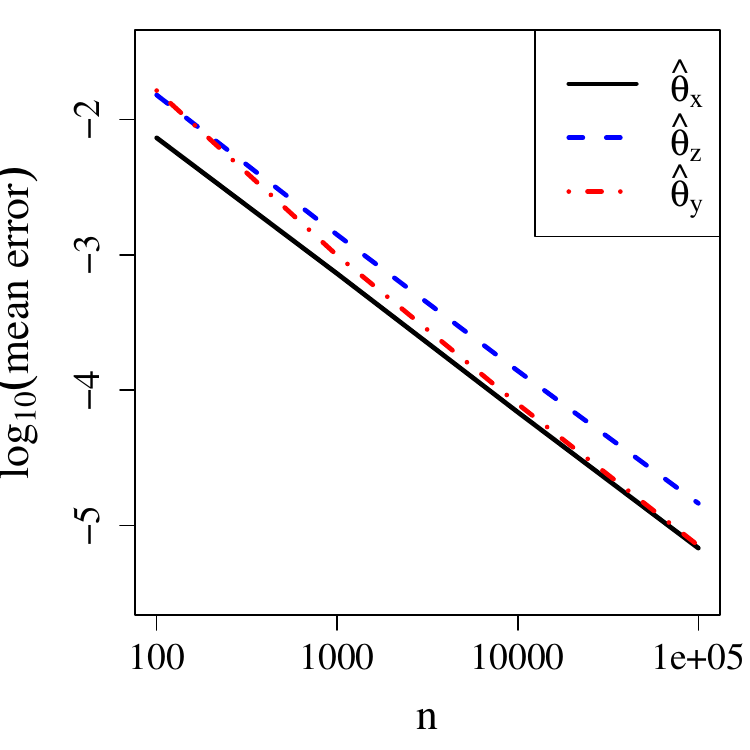}
		\caption{Simulations  corresponding to the log-linear model with two-way interactions from Section  \ref{s:log-linear}}
		\label{fig:log-linear}
		\end{subfigure}
		\hspace{.02\textwidth}
		\begin{subfigure}[t]{0.48\textwidth}
			\includegraphics[width=\textwidth]{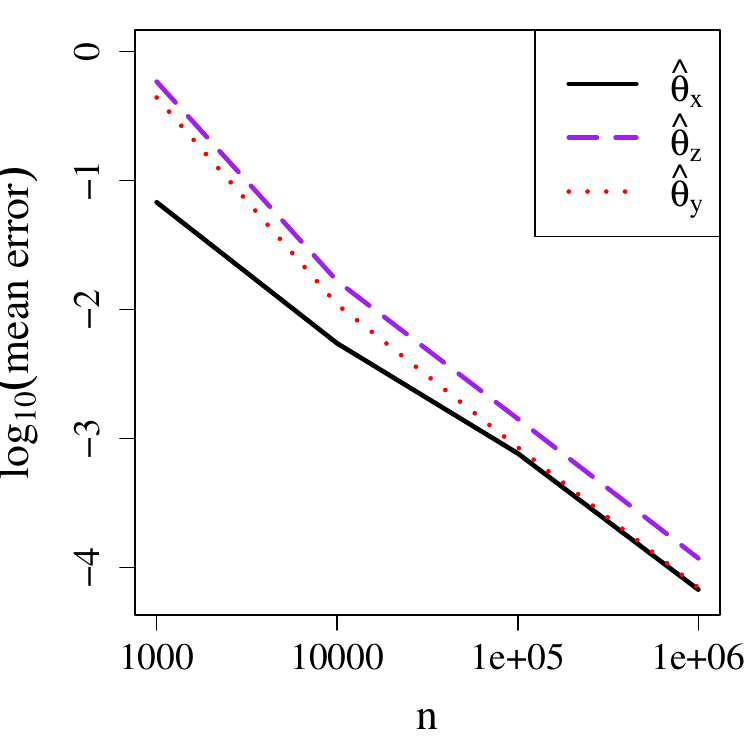}
			\caption{Simulations for the beta distribution from Section \ref{s:beta}. $\thx$ is the MLE. $\thz$ and $\thy$ both satisfy 1-DP.}
			\label{fig:beta}
			\end{subfigure}
			\caption{Average squared $\ell_2$-distance between the estimated parameters and the true parameters on the log-scale. Averages are over 200 replicates for both plots. $\thx$ is from the true model, $\thz$ from the fitted model, and $\thy$ from Algorithm \ref{alg:onestep}.}
		\end{figure}

		\begin{equation}\label{eq:loglinear}
		\log \mu_{ijk\ell} = \la + \la_i^G + \la_j^L +\la_k^S + \la_\ell^I
		+\la_{ij}^{GL}+\la_{ik}^{GS}+\la_{i\ell}^{GI}
		+\la_{jk}^{LS}+\la_{j\ell}^{LI}+\la_{k\ell}^{SI}\end{equation}
		
		For our simulations, we	treat the fitted parameters as the true parameters, to ensure that model assumptions are met. We simulate from the fitted model at sample sizes $n\in \{10^2,10^3,10^4,10^5\}$ and compare the performance in terms of the fitted probabilities for each bin of the contingency table. 		The results are plotted in Figure \ref{fig:log-linear}, with both axes on log-scale. The ``mean error'' is the average squared $\ell_2$ distance between the estimated parameter vector and the true parameter vector, averaged over 200 replicates. To interpret the plot, note that if the error is of the form $\mathrm{error} = cn^{-1}$, where $c$ is a constant, then $\log(\mathrm{error}) = c+(-1) \log(n)$. So, the slope represents the convergence rate, and the vertical offset represents the asymptotic variance. In Figure \ref{fig:log-linear}, we see that the curve for $\thy$, based on one-step samples,  approaches the curve for $\thx$, indicating that they have the same asymptotic rate and variance. On the other hand, the curve for  $\thz$, based on parametric bootstrap samples, has the same slope, but does not approach the $\thx$ curve, indicating that $\thz$ has the same rate but inflated variance. { In fact, Theorem \ref{thm:TVnaive} indicates that $\var(\thz)\approx 2\var(\thx)$.}

		We see that our procedure approximately preserves the sufficient statistics, similar to sampling from a conditional distribution. 	Previous work has proposed procedures to sample directly from conditional distributions for contingency table data. However, these approaches require sophisticated tools from algebraic statistics, and are computationally expensive (e.g., MCMC) \citep{karwa2013conditional}. In contrast, our approach is simple to implement and highly computationally efficient. Our approach is also applicable for a wide variety of models, whereas the techniques to sample directly from the conditional distribution often require a tailored approach for each setting.
		
\subsection{Differentially private beta distributed synthetic data}\label{s:beta}

		In this example, we assume that $X_1,\ldots, X_n \iid \mathrm{Beta}(\alpha,\beta)$, where $\alpha,\beta\geq 1$, and our goal is to produce differentially private (DP) synthetic data. 	Often, to ensure finite sensitivity, the data are clamped to artificial bounds $[a,b]$, introducing bias in the DP estimate. Naive bounds are fixed in $n$, resulting in asymptotically negligible noise, but $O_p(1)$ bias. In Section S2 of the Supplementary Material, we show that it is possible to increase the bounds in $n$ to produce both noise and bias of order $o_p(n^{-1/2})$, resulting in an efficient DP estimator. In this section, we show through simulations that using this estimator along with Algorithm \ref{alg:onestep} results in a DP sample with optimal asymptotics. 

		For the simulation, we sample $X_1,\ldots, X_n \iid \mathrm{Beta}(5,3)$, with varying sample sizes, $n \in \{10^3,10^4,10^5,10^6\}$. We estimate $\thx$ with the MLE. Using $\epsilon=1$, we  clamp and add Laplace noise to the sufficient statistics to obtain our privatized summary statistics (see Section S2 of the Supplementary Material for details), 
		and then obtain $\hat \theta_{DP}$  by maximizing the log-likelihood with the privatized values plugged in for the sufficient statistics. We sample $Z_1,\ldots, Z_n \iid f_{\hat \theta_{DP}}$ and calculate the MLE $\thz$. We draw $(Y_i)_{i=1}^n$ from Algorithm \ref{alg:onestep} using $\hat\theta_{DP}$ in place of $\thx$.  In Figure \ref{fig:beta}, we plot the mean squared $\ell_2$ error between each estimate of $\theta$ from the true value $(5,3)$, over 200 replicates. From the log-scale plot, we see that $\hat\theta_{DP}$ and $\thy$ have the same asymptotic performance as the MLE, whereas $\thz$ has inflated variance. See Example \ref{s:log-linear} for more explanation of this interpretation. Section S.4.1 of the Supplementary Material contains additional simulations where the value of $\epsilon$ is varied. 
		
		\subsection{DP two sample proportion test}\label{s:DPtesting}

		\begin{figure}[t]
		\begin{subfigure}[t]{0.48\textwidth}
			\includegraphics[width=\textwidth]{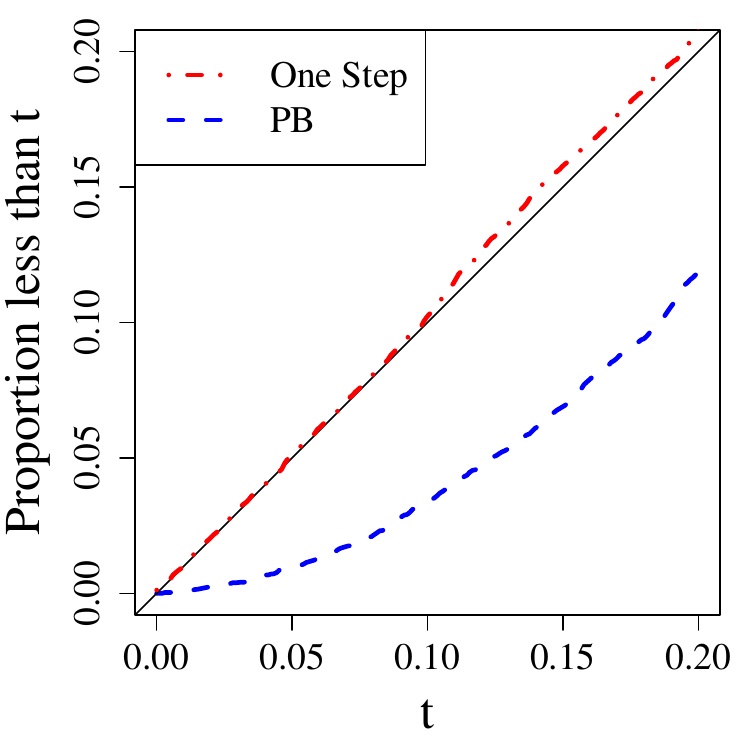}
		\caption{Under the null  hypothesis, the empirical cumulative distribution of both tests. $\theta_X=\theta_Y=.3$. Results based on 10,000 replicates.  }
			\label{fig:pvalue}
			\end{subfigure}
				\hspace{.02\textwidth}
		  \begin{subfigure}[t]{0.48\textwidth}
		\includegraphics[width=\textwidth]{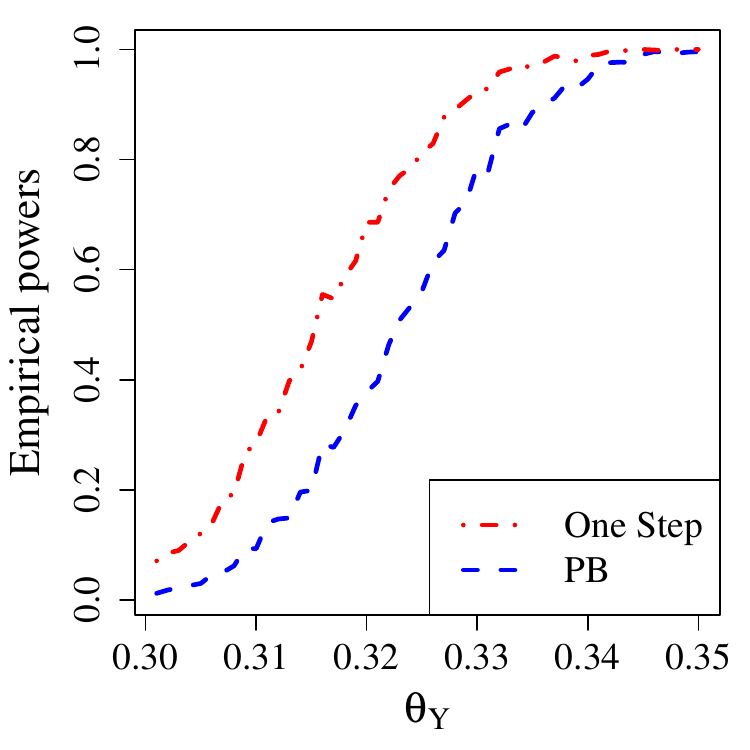}
		\caption{Empirical power at type I error $.05$. $\theta_X=.3$, $\theta_Y$ is varied along the $x$-axis. Results averaged from 10,000 replicates for each value of $\theta_Y$.}
		\label{fig:power}
		\end{subfigure}
		\caption{Simulations for the DP two sample proportion test of Section \ref{s:DPtesting}. In red is the one-step test, and in blue is the parametric bootstrap test (abbreviated as PB). Sample sizes are $n=m=200$, privacy parameter is $\ep=1$, and type I error is $.05$.}
		\label{DP2prop}
		\end{figure}

    In this section, we illustrate how the one-step samples can be used to perform approximate hypothesis tests. We base the simulation on a problem in differential privacy, where we are given only access to DP summary statistics and are tasked with testing a hypothesis on the generating distribution for the (missing) private data. Such settings result in complex distributions which can be difficult to work with directly, making MCMC methods such as in \citet{barber2020testing} cumbersome and potentially intractable. We show through a simulation that the one-step samples give highly accurate $p$-values and improved power compared to the parametric bootstrap. 
	
		Suppose we have two independent samples of binary data, one from a ``control population'' and another from a ``treatment population.'' We denote $X_1,\ldots, X_n\iid \mathrm{Bern}(\theta_X)$ as the control sample, and $Y_1,\ldots, Y_m \iid \mathrm{Bern}(\theta_Y)$ as the treatment sample. Note that $m$ used in this example is not related to the dimension of $\Omega$, as stated in (R0).
		
		It was shown in \citet{awan2018differentially} and \citet{awan2020differentially} that the \emph{Tulap} distribution is the optimal DP mechanism to generate uniformly most powerful hypothesis tests and uniformly most accurate confidence intervals for Bernoulli data. Thus, to satisfy $\ep$-differential privacy, the data curators release the following noisy statistics: $\twid X = \sum_{i=1}^n X_i + N_1$ and $\twid Y = \sum_{i=1}^m Y_i + N_2$, where $N_1,N_2 \iid \mathrm{Tulap}\{0,\exp(-\ep),0\}$ and the sample sizes $m$ and $n$ are released without modification. { Recall that in the $\mathrm{Tulap}(m,b,q)$ distribution, $m\in \RR$ is a location parameter, $b\in (0,\infty)$ is related to the scale/shape with higher values increasing the dispersion, and $q$ is a truncation parameter with $0$ indicating that no truncation takes place. In our case, }
		$N_i \overset d=  G_1-G_2+U$, where $G_1,G_2 \iid \mathrm{Geom}\{1-\exp(-\ep)\}$ and $U\sim \mathrm{Unif}(-1/2,1/2)$. We can think of $\twid X$ and $\twid Y$ as noisy counts for the control and treatment groups, respectively. 
		
		Based only on the privatized summary statistics $\twid X$ and $\twid Y$, we are to test  $H_0: \theta_X= \theta_Y$ versus $H_1: \theta_X{ <} \theta_Y$. Without privacy, there exists a uniformly most powerful test, which is constructed by conditioning on the total number of ones: $\sum_{i=1}^n X_i + \sum_{j=1}^m Y_j$, a complete sufficient statistic under the null hypothesis. However, with the noisy counts, it can be verified that there is not a low-dimensional sufficient statistic. On the other hand, an efficient estimator for $\theta_X=\theta_Y$ under the null hypothesis is $\hat\theta(\twid X,\twid Y) = \min[\max\{(\twid X + \twid Y)/(m+n),0\},1]$. Note that deriving the exact distribution of $(\twid X,\twid Y)\mid \hat\theta(\twid X,\twid Y)$ is fairly complex, involving the convolution of distributions. However, the one-step method can easily produce approximate samples from this conditional distribution. In what follows, we use the one-step algorithm, given in Algorithm \ref{alg:onestep} and investigate the properties of a hypothesis test based on this conditional distribution in comparison with a parametric bootstrap test.
		
		Recall that without privacy, the uniformly most powerful test uses the test statistic $Y$, and threshold computed from the conditional distribution of $Y|X+Y$ under the null hypothesis. With privacy, we use the test statistic $\twid Y$,  the noisy count of ``ones'' in the treatment group, and compute the $p$-values based on $\twid Y \mid \hat \theta(\twid X,\twid Y)$. In particular, we compare the performance of this test versus the parametric bootstrap test, which uses the test statistic $\twid Y$ based on the approximate sampling distribution, which is the convolution of $\mathrm{Binom}\{m,\hat\theta(\twid X,\twid Y)\}$ and $\mathrm{Tulap}\{0,\exp(-\ep),0\}$. 
		
		For the simulation, we use the sample size $n=m=200$, fix $\theta_X=0.3$, set the privacy parameter to $\ep=1$, and base the simulation on 10,000 replicates.  Under the null hypothesis, where $\theta_Y=\theta_X$, we plot the empirical cumulative distribution function (CDF) for the $p$-values of the  proposed test as well as for the parametric bootstrap test in Figure \ref{DP2prop}(a). Recall that a properly calibrated $p$-value will have the CDF of $U(0,1)$. We see that the empirical CDF of the $p$-values for the one-step test closely approximate the ideal CDF, whereas the $p$-values of the parametric bootstrap test are overly conservative.
		
		Next, we study the power of the one-step test versus the parametric bootstrap test in Figure \ref{DP2prop}(b). For this simulation, we set $n=m=200$, fix $\theta_X=.3$, and set $\ep=1$. We vary the value of $\theta_Y$ along the $x$-axis by increments of $.001$ and plot the empirical power of the two tests, averaged over 10,000 replicates for each value of $\theta_Y$. We see that the one-step test offers a considerable increase in power over the bootstrap test. Section S.4.2 of the Supplementary Material contains additional simulations where the value of $\epsilon$ is varied.

	\section{Discussion} \label{s:conclusions}
	We proposed a simple method of producing synthetic data from a parametric model, which approximately preserves efficient statistics.  We also provided evidence in Section \ref{s:distribution} that the one-step synthetic data results in a distribution which approaches the true underlying distribution. Both of these properties are in contrast to the common approach of sampling from a fitted model, which we showed results in inefficient estimators and inconsistent synthetic data. Our one-step approach is also widely applicable to parametric models and is both easily implemented and highly computationally efficient. It also allows for both partially synthetic data, as well as differentially private fully synthetic data by incorporating a DP efficient estimator. 
    
     Besides synthetic data, there is also promise for using the one-step approach for hypothesis tests as well. \citet{barber2020testing} showed that one can produce powerful and accurate hypothesis tests by conditioning on an efficient statistic for the null model. However, their approach is likelihood-based and requires MCMC methods for implementation. In problems with high-dimensional latent variables such an approach is inapplicable, for example in differential privacy where the entire private database is latent.  The one-step approach offers a computationally efficient alternative, which we demonstrated in Section \ref{s:DPtesting} gives a more accurate and powerful test than the parametric bootstrap. 

    
    In Section \ref{s:distribution}, we studied the distributional properties of a modified version of the one-step method, which suggested that the one-step synthetic data converges to the true joint distribution as the sample size increases. 	We also saw in Section  \ref{s:burr} that in the case of the Burr distribution, the Kolmogorov-Smirnov test cannot distinguish between our synthetic sample and the true distribution which generated the original sample, supporting the conjecture that the one-step sample is consistent. In future work, we propose to formally prove that the total variation distance between the one-step samples { (rather than an approximation) and the true sampling distribution} goes to zero as the sample sizes goes to infinity. This will provide additional theoretical justification for using the one-step method for synthetic data and hypothesis tests. { Future researchers may investigate whether the data-augmentation MCMC algorithm of \citet{ju2022data} could enable the approach of \citet{barber2020testing} to be applied to DP problems.}
    
    While our approach was focused on parametric models, similar to Theorem \ref{thm:TVnaive}, there is a  loss of ``efficiency'' when sampling from a non-parametric model as well. Similar to the approach in this paper, it may be advantageous to sample from a non-parametric model conditional on the sample having ``similar'' estimates as the original data. An interesting future direction would be to formalize and investigate this direction.

    

    
    \section*{Supplementary Material}
    Section S1 of the Supplementary Materials includes a brief introduction to differential privacy. Section S2 gives details for the derivation of the privatized beta estimates of Section \ref{s:beta}. Proofs and technical details are provided in Section S3. Section S4 contains some additional simulation results.
    
    \section*{Acknowledgments}
    { The first author was supported in part by NSF grants SES-1853209 and SES-2150615.} The authors are very grateful to Roberto Molinari, Matthew Reimherr, and Aleksandra Slavkovi\'c for several helpful discussions, as well as feedback on early drafts.

\bibliographystyle{chicago}      
\bibliography{bibliography}   

\centerline{\bf Supplementary Material}
\vspace{.55cm}
\fontsize{9}{11.5pt plus.8pt minus .6pt}\selectfont
\noindent

\par

\setcounter{section}{0}
\setcounter{equation}{0}
\def\theequation{S\arabic{section}.\arabic{equation}}
\def\thesection{S\arabic{section}}

\fontsize{12}{14pt plus.8pt minus .6pt}\selectfont
\renewcommand\thefigure{S.\arabic{figure}}    
Section S1 includes a brief introduction to differential privacy. Section S2 gives details for the derivation of the privatized beta estimates of Section 6.3. Proofs and technical details are provided in Section S3. Section S4 contains some additional simulation results.
	\section{Background on Differential Privacy}\label{s:dp}
	
	In this section, we review the basics of differential privacy (DP), which was proposed by \citet{dwork2006calibrating} as a framework to mathematically quantify the degree of privacy protection. To satisfy differential privacy, a method must introduce additional randomness into the analysis, and the constraint of DP requires that for all possible databases, the change in one person's data does not significantly change the distribution of outputs. Consequently, having observed the DP output, an adversary cannot accurately determine the input value of any single person in the database. Definition \ref{def:dp} gives a formal definition of DP. In Definition \ref{def:dp}, $h: \mscr X^n \times \mscr X^n \rightarrow \ZZ^{\geq 0}$ represents the \emph{Hamming metric}, defined by $h(\ul x,\ul x') =\# \{i \mid x_i\neq x'_i\}$. 
	
	
	\begin{definition}
		[Differential privacy: \citet{dwork2006calibrating}]\label{def:dp}
		Let the \emph{privacy parameter} $\ep>0$ and the sample size $n\in \{1,2,\ldots\}$ be given. Let $\mscr X$ be any set, and $(\mscr Y, \mscr S)$  a measurable space. Let $\mscr M = \{M_{\ul x} \mid \ul x\in \mscr X^n\}$ be a set of probability measures on $(\mscr Y, \mscr S)$, which we call a \emph{mechanism}. We say that $\mscr M$ satisfies \emph{$\ep $-differential privacy} ($\ep$-DP) if $M_{\ul x}(S) \leq e^\ep M_{\ul x'}(S)$ for all $S\in \mscr S$ and all $\ul x,\ul x' \in \mscr X^n$ such that $h(\ul x,\ul x')=1$.
	\end{definition}

	An important property of differential privacy is that it is invariant to post-processing. Applying any data-independent procedure to the output of a DP mechanism preserves $\ep$-DP \citep[Proposition 2.1]{dwork2014algorithmic}. Furthermore, \citet{smith2011privacy} demonstrated that under conditions similar to (R1)-(R3), there exist efficient DP estimators for parametric models. Using these techniques, the one-step procedure can produce DP synthetic data by using a DP efficient statistic.

	\begin{rem}
		Besides Definition \ref{def:dp}, there are many other variations of differential privacy, the majority of which are relaxations of Definition \ref{def:dp}, which also allow for efficient estimators. For instance, approximate DP \citep{dwork2006our}, concentrated DP \citep{Dwork2016:ConcentratedDP,Bun2016ConcentratedDP}, truncated-concentrated DP \citep{bun2018composable}, Renyi DP \citep{Mironov2017}, and Gaussian DP \citep{dong2022gaussian} all allow for efficient estimators. On the other hand, local differential privacy \citep{kasiviswanathan2011,Duchi2013FOCS} in general does not permit efficient estimators and would not fit in our framework.  For an axiomatic treatment of formal privacy, see \citet{Kifer2012axiomatic}.
	\end{rem}

	One of the earliest and simplest privacy mechanisms is the \emph{Laplace mechanism}. 
	Given a statistic $T$, the Laplace mechanism adds independent Laplace noise to each entry of the statistic, with scale parameter proportional to the \emph{sensitivity} of the statistic. Informally, the sensitivity of $T$ is the largest amount that $T$ changes, when one person's data is changed in the dataset. 
	
	\begin{prop}[Sensitivity and Laplace Mechanism: \citet{dwork2006calibrating}]\label{prop:laplace}
		Let the privacy parameter $\ep>0$ be given, and let $T: \mscr X^n \rightarrow \RR^p$ be a statistic. The \emph{$\ell_1$-sensitivity} of $T$ is  
		$\Delta_n(T) = \sup \lVert T(\ul x) - T(\ul x')\rVert_1,$ 
		where the supremum is over all $\ul x,\ul x'\in \mscr X^n$ such that $h(\ul x,\ul x')=1$.  
		Provided that $\Delta_n(T)$ is finite, 
		releasing the vector $\left\{T_j(\ul x) + L_j\right\}_{j=1}^p$ satisfies $\ep$-DP, where 
		$L_1,\ldots, L_p \iid \mathrm{Laplace}\left\{\Delta_n(T)/\ep\right\}$.
	\end{prop}
	
	\section{Deriving an Efficient DP Estimator for the Beta Distribution}\label{s:betaProof}

		We assume that $X_1,\ldots, X_n \iid \mathrm{Beta}(\alpha,\beta)$, where $\alpha,\beta\geq 1$, and our goal is to produce differentially private (DP) synthetic data.  Recall that $X_i$ takes values in $[0,1]$ and has pdf $f_X(x) = x^{\alpha-1}(1-x)^{\beta-1}/B(\alpha,\beta)$, where $B$ is the Beta function.

		Often, to ensure finite sensitivity, the data are clamped to artificial bounds $[a,b]$, introducing bias in the DP estimate. Naive bounds are fixed in $n$, resulting in asymptotically negligible noise, but $O_p(1)$ bias. However, we show that it is possible to increase the bounds in $n$ to produce both noise and bias of order $o_p(n^{-1/2})$, resulting in an efficient DP estimator. We show through simulations that using this estimator along with Algorithm 1 results in a DP sample with optimal asymptotics. While we work with the beta distribution, this approach may be of value for other exponential family distributions as well.  We note that the asymptotics of clamping bounds have appeared in other DP works, but which are not immediately applicable to our setting (e.g., \citealp{smith2011privacy,kamath2020private}).

		Recall that  $n^{-1} \sum_{i=1}^n \log(X_i)$ and $n^{-1}\sum_{i=1}^n \log(1-X_i)$ are sufficient statistics for the beta distribution. We will add Laplace noise to each of these statistics to achieve differential privacy. However, the sensitivity of these quantities is unbounded. First we pre-process the data by setting 
		$\twid X_i = \min\{\max(X_i,t),1-t\},$ 
		where $t$ is a threshold that depends on $n$. Then the $\ell_1$-sensitivity of the pair of sufficient statistics is
		$\Delta(t) =2n^{-1}\left|\log(t)-\log(1-t)\right|$. 
		We add independent noise to each of the statistics from the distribution $\mathrm{Laplace}\{\Delta(t)/\epsilon\}$, which results in $\ep$-DP versions of these statistics. Finally, we estimate $\theta=(\alpha,\beta)$ by plugging in the privatized sufficient statistics into the log-likelihood function and maximizing over $\theta$. The resulting parameter estimate satisfies $\ep$-DP by post-processing. 
		
		We must carefully choose the threshold $t$ to ensure that the resulting estimate is efficient. The choice of $t$ must satisfy $\Delta(t) = o(n^{-1/2})$ to ensure that the noise does not affect the asymptotics of the  likelihood function. We also require that both $P(X_i<t)=o(n^{-1/2})$, and $P(X_i>1-t)=o(n^{-1/2})$ to ensure that $\twid X_i =X_i + o_p(n^{-1/2})$, which limits the bias to  $o_p(n^{-1/2})$. For the beta distribution, we can calculate that $P(X_i<t) = O(t^\alpha)$ and $P(X_i>1-t) = O(t^\beta)$. Since we assume that $\alpha,\beta\geq 1$, so long as $t=o(n^{-1/2})$ the probability bounds will hold. Taking $t = \min[1/2,10/\{\log(n) \sqrt n\}]$ satisfies $t=o(n^{-1/2})$, and we estimate the sensitivity as
		\[\Delta(t)\leq 2n^{-1} \log(t^{-1})\leq 2n^{-1}\log\{\log(n)\sqrt n\} = O\{\log(n)/n\} = o(n^{-1/2}),\] 
		which satisfies our requirement for $\Delta$. While there are other choices of $t$ which would satisfy the requirements, our threshold was chosen to optimize the finite sample performance, so that the asymptotics could be illustrated with smaller sample sizes.
	
	\section{Proofs and Technical Lemmas}\label{s:proofs}

	For two distributions $P$ and $Q$ on $\RR^k$, the \emph{Kolmogorov-Smirnov distance} (KS-distance) is $\mathrm{KS}(P,Q) = \sup_{R\text{ rectangle}}|P(R)-Q(R)|$, where the supremum is over all axis-aligned rectangles. If $F_P$ and $F_Q$ are the multivariate cdfs of $P$ and $Q$, then $\lVert F_P-F_Q\rVert_\infty \leq \mathrm{KS}(P,Q)\leq 2^k\lVert F_P-F_Q\rVert_\infty$, so convergence in distribution is equivalent to convergence in KS-distance \citep{smith2011privacy}. By definition, we have that $\mathrm{TV}(P,Q)\geq \mathrm{KS}(P,Q)$.

	\begin{proof}[Proof of Theorem 1]
    Call $H_n(\theta)$ be the distribution of $\sqrt n\{\hat\theta(\ul Y)-\theta\}$ when $\hat\theta(\ul Y)$ is based on $Y_1,\ldots,Y_n\iid f_\theta$, and note that because $\hat\theta$ is an efficient estimator, we have that $H_n(\theta)\rightarrow N\{0,I^{-1}(\theta)\}$. 
    By the Skorohod Representation Theorem \citep[Section 1.6.3]{serfling2009approximation}, there exists $U_1\sim U(0,1)$ and measurable functions $A, A_n:[0,1] \rightarrow \RR^d$ for all $n\in\mathbb{N}$ such that     
    1) $A_n(U_1)\sim H_n(\theta_0)$ for all $n\in \mathbb{N}$, 2) $A(U_1) \sim N\{0, I^{-1}(\theta_0)\}$, and 3) $A_n(U_1) \overset {a.s.}\rightarrow A(U_1)$. 
    This implies that $A_n(U_1) \overset d = \sqrt n\{\hat\theta(\ul X)-\theta_0\}$ or equivalently that $\theta_0+A_n(U_1)/\sqrt n \overset d =\hat\theta(\ul X)$.

Call $S=\{u_1\in [0,1]| A_n(u_1)\rightarrow A(u_1)\}$, which satisfies $P(U_1\in S)=1$. Then for any $u_1\in S$, $A_n(u_1)$ is a convergent sequence which may be used in the role of $h_n$ in the definition of LAE. So, using the fact that $\hat\theta$ is LAE, we will apply the Skorohod Representation Theorem to the sequence of conditional random variables,
\[\sqrt n \{\hat \theta(\ul Z)-\hat\theta(\ul X)\}|\{\hat\theta(\ul X)=\theta_0+A_n(u_1)/\sqrt n\}\sim H_n\{\theta_0+A_n(u_1)/\sqrt n\}.\]
Let $U_2\sim U(0,1)$, independent of $U_1$. For all $u_1\in S$, there exists measurable functions $B,B^{A_n(u_1)}_n:[0,1]\rightarrow \RR^d$ for all $n\in \mathbb{N}$ which satisfy  
1) $B_n^{A_n(u_1)}(U_2) \sim H_n\{\theta_0+A_n(u_1)/\sqrt n\}$ for all $n \in \mathbb{N}$,
    2)  $B(U_2) \sim N\{0,I^{-1}(\theta_0)\}$, and 
    3) $B_n^{A_n(u_1)}(U_2) \overset {a.s.}\rightarrow B(U_2)$. 
Note that since $P(U_1\in S)=1$, the above points about the $B_n$'s still hold when replacing $u_1$ with $U_1$ and considering the randomness over both $U_1$ and $U_2$. 

By construction, we have that $A_n(U_1)\overset d = \sqrt n \{\hat\theta(\ul X)-\theta_0\}$ as well as $B_n^{A_n(u_1)}(U_2) \overset d = \sqrt n\{\hat\theta(\ul Z)-\hat\theta(\ul X)\}|\left\{\hat\theta(\ul X)=\theta_0+A_n(u_1)/\sqrt n\right\}$. Thus, the following joint distributions agree:
\[\left(\begin{array}{c} A_n(U_1)\\ B_n^{A_n(U_1)}(U_2)\end{array}\right)
\overset d=\left(\begin{array}{c} \sqrt n \{\hat \theta(\ul X)-\theta_0\}\\ \sqrt n \{\hat\theta(\ul Z)-\hat\theta(\ul X)\}\end{array}\right).\]
By the continuous mapping theorem, we have that 
\begin{align*}
    \sqrt n (\hat\theta(\ul Z)-\theta_0)&=\sqrt n\{\hat\theta(\ul X)-\theta_0\}+\sqrt n\{\hat\theta(\ul Z)-\hat\theta(\ul X)\}\\
    &\overset d= A_n(U_1)+B_n^{A_n(U_1)}(U_2)\\
					&\overset{a.s.}\rightarrow A(U_1)+B(U_2)\\
					&\overset d= N\{0,2I^{-1}(\theta_0)\},
\end{align*}
where in the last step, we used the fact that $A(U_1),B(U_2) \iid N\{0,I^{-1}(\theta_0)\}$, since $U_1$ and $U_2$ are independent. This concludes the proof for part 1.


For part 2, we lower bound the total variation distance between the distributions of $\ul X$ and $\ul Z$ as follows:
\begin{align}
\mathrm{TV}\left(\ul X,\ul Z\right)&\geq \mathrm{TV}\left\{\sqrt n(\thx-\theta_0),\sqrt n (\thx-\theta_0)\right\}\label{eq:process}\\
&\geq \mathrm{KS}\left\{\sqrt n(\thx-\theta_0),\sqrt n (\thx-\theta_0)\right\}\label{eq:ks}\\
&\geq\mathrm{KS}\left[N\{0,I^{-1}(\theta_0)\},N\{0,2I^{-1}(\theta_0)\}\right]\\
&\phantom{=}-\mathrm{KS}\left[\sqrt n (\thx-\theta_0),N\{0, I^{-1}(\theta_0)\}\right]\\
&\phantom{=} -\mathrm{KS}\left[\sqrt n(\thz-\theta_0),N\{0,2I^{-1}(\theta_0)\}\right]\label{eq:triangle}\\
&= \mathrm{KS}\left[N\{0,I^{-1}(\theta_0)\},N\{0,2I^{-1}(\theta_0)\}\right]+o(1)\label{eq:convD}\\
&\geq \Phi\left\{-\sqrt{\log(4)}/\sqrt 2\right\}-\Phi\left\{-\sqrt{\log(4)}\right\}+o(1)\label{eq:lower}\\
&\geq .083+o(1)
\end{align}
where \eqref{eq:process} is by the data processing inequality, \eqref{eq:ks} uses the KS-distance as a lower bound on total variation, \eqref{eq:triangle} applies two triangle inequalities since KS-distance is a metric, and \eqref{eq:convD} uses the asymptotic distributions of $\thx$ and $\thz$. 

To establish \ref{eq:lower}, consider the following. Denote $\sigma^2 = (I^{-1}(\theta_0))_{1,1}$. Then consider the sequence of rectangles $R_i = \{x\in \RR^k \mid -(i+1)\sigma\leq x_1\leq -\sqrt{\log(4)}\sigma, \text{ and }-i\leq x_j\leq i, \forall j\neq 1\}$. Note that $R_i\subset R_{i+1}$ and $\bigcup_{i=1}^\infty R_i = \{x\in \RR\mid x_1\leq -\sqrt{\log(4)} \sigma\}$. Denote by $P$ the probability measure for $N\{0,2I^{-1}(\theta_0)\}$ and $Q$ the probability measure for $N\{0,I^{-1}(\theta_0)\}$. Then 

\begin{align*}
    \mathrm{KS}\left[N\{0,I^{-1}(\theta_0)\},N\{0,2I^{-1}(\theta_0)\}\right]
    &\geq \lim_{i\rightarrow \infty} |P(R_i)-Q(R_i)|\\
    &=\left|\lim_{i\rightarrow \infty} P(R_i)-\lim_{i\rightarrow\infty}Q(R_i)\right|\\
    &=\left|P\left(\bigcup_{i=1}^\infty R_i\right)-Q\left(\bigcup_{i=1}^\infty R_i\right)\right|\\
    &= \Phi\left\{-\sqrt{\log(4)}/\sqrt 2\right\}-\Phi\left\{-\sqrt{\log(4)}\right\}\\
&\geq .083,
\end{align*}
where the value $\sqrt{\log(4)}$ was chosen as it is the maximizer of $\Phi(-t/\sqrt{2})-\Phi(t)$.
\end{proof}

For the following proofs, we will overload the $\frac{d}{d\theta}$ operator when working with multivariate derivatives. For a function $f:\RR^p\rightarrow \RR$, we write $\frac{d}{d\theta} f(\theta)$ to denote the $p\times 1$ vector of partial derivatives $(\frac{\partial}{\partial \theta_j} f(\theta))_{j=1}^p$. For a function $g:\RR^p\rightarrow \RR^q$, we write $\frac{d}{d\theta} g(\theta)$ to denote the $p\times q$ matrix $(\frac{\partial}{\partial \theta_j} g_k(\theta))_{j,k=1}^{p,q}$.

	Lemmas \ref{lem:score} and \ref{lem:fisher} are used for the proof of Theorem 2. 		Parts 1 and 2 of Lemma \ref{lem:score} can be rephrased as the following: $\hat \theta$ is efficient if and only if it is consistent and $n^{-1}\sum_{i=1}^n S(\hat \theta,X_i) = o_p(n^{-1/2})$. The third property of Lemma \ref{lem:score} is similar to many standard expansions used in asymptotics, for example in \citet{van2000asymptotic}. However, we require the expansion for arbitrary efficient estimators, and include a proof for completeness.
	\begin{lemma}\label{lem:score}
		Suppose $X_1,\ldots, X_n\iid f_{\theta_0}$, and assume that (R1)-(R3) hold. Let $\hat \theta_X$ be an efficient estimator, which is a sequence of zeros of the score  function. Suppose that $\twid \theta_X$ is a $\sqrt n$-consistent estimator of $\theta_0$. Then 
		\begin{enumerate}
			\item If $n^{-1} \sum_{i=1}^n S(\twid \theta_X, X_i) = o_p(n^{-1/2})$, then $\twid \theta_X - \hat \theta_X = o_p(n^{-1/2})$.
			\item If $\twid \theta_X$ is efficient, then $n^{-1}  \sum_{i=1}^n S(\twid \theta_X, X_i) = o_p(n^{-1/2})$.
			\item If $\twid \theta_X$ is efficient, then $\twid \theta_X = \theta_0 + I^{-1}(\theta_0) n^{-1} \sum_{i=1}^n S(\theta_0,X_i) + o_p(n^{-1/2})$.
		\end{enumerate}
	\end{lemma}
	\begin{proof}
		As $\twid \theta_X$ and $\hat \theta_X$ are both $\sqrt n$-consistent, we know that $\twid \theta_X -\hat \theta_X = O_p(n^{-1/2})$. So, we may consider a Taylor expansion of the score  function about $\twid \theta_X = \hat \theta_X$.
		\begin{equation}\begin{split}\label{eq:efficient}
		&n^{-1} \sum_{i=1}^n S(\twid \theta_X, X_i) \\
  &= n^{-1} \sum_{i=1}^n S(\hat \theta_X,X_i) 
  + \left\{\frac{d}{d \theta} n^{-1} \sum_{i=1}^n S(\theta,X_i)\Big|_{\theta=\thx}\right\}(\twid \theta_X - \hat \theta_X) + O_p(n^{-1}) \\
    		&= 0+\left\{ \frac{d}{d\theta} n^{-1} \sum_{i=1}^n S(\theta, X_i)\big|_{\theta=\thx} +  O_p(n^{-1/2})\right\}(\twid \theta_X - \hat \theta_X)\\
		&=\left\{-I(\theta_0) + o_p(1)\right\}(\twid \theta_X - \hat \theta_X),
		\end{split}\end{equation}
		where we used assumptions (R1)-(R3) to justify that 1) the second derivative is bounded in a neighborhood about $\theta_0$ (as both $\hat \theta_X$ and $\twid \theta_X$ converge to $\theta_0$), 2) the derivative of the score converges to $-I(\theta_0)$ by \citet[Theorem 7.2.1]{lehmann2004elements} along with the Law of Large Numbers, and 3) that $I(\theta_0)$ is finite, by (R3). 
		
		To establish property 1, note that the left hand side of Equation \eqref{eq:efficient} is $o_p(n^{-1/2})$ implying that $(\twid \theta_X - \hat \theta_X) = o_p(n^{-1/2})$. For property 2, recall that by \citet[page 479]{lehmann2004elements}, if $\twid \theta_X$ and $\hat \theta_X$ are both efficient, then $(\twid \theta_X - \hat \theta_X) = o_p(n^{-1/2})$. Plugging this into the right hand side of Equation \eqref{eq:efficient} gives $n^{-1} \sum_{i=1}^n S(\twid \theta_X, X_i) = o_p(n^{-1/2})$, establishing property 2.

		For property 3, we consider a slightly different expansion:
		\begin{align*}
		    o_p(n^{-1/2})& = n^{-1} \sum_{i=1}^n S(\twid \theta, X_i)\\ 
		&= n^{-1} \sum_{i=1}^n S(\theta_0,X_i) + \frac d{d\theta_0} n^{-1}\sum_{i=1}^n S(\theta_0,X_i) (\twid \theta - \theta_0) + O_p(n^{-1}),\\
		&=n^{-1} \sum_{i=1}^n S(\theta_0,X_i) +\{- I(\theta_0)+o_p(1)\}(\twid \theta - \theta_0) + O_p(n^{-1})
		\end{align*}
		where we used property 2 for the first equality, expanded the score about $\hat\theta_X = \theta_0$ for the second, and justify the $O_p(n^{-1})$ by (R2). By (R1)-(R2) and Law of Large Numbers along with \citet[Theorem 7.2.1]{lehmann2004elements}, we have the convergence of the derivative of score to $-I(\theta_0)$. By (R3), $I(\theta_0)$ is invertible. Solving the equation for $\twid \theta_X$ gives the desired result.
	\end{proof}
	
		\begin{lemma}\label{lem:fisher}
		Assume that (R0)-(R4) hold, and let $\omega_1,\ldots, \omega_n \iid P$. Then 
		\begin{align*}
		n^{-1} \sum_{i=1}^n \frac{d}{d\theta}  S\{\theta,X_\theta(\omega_i)\}= o_p(1).
		\end{align*}
	\end{lemma}
	 
	\begin{proof}
	First we can express the derivative as 
	\begin{align*}
	 &   n^{-1} \sum_{i=1}^n \frac{d}{d\theta}S\{\theta,X_\theta(\omega_i)\}\\
	&=n^{-1} \sum_{i=1}^n \left[ \frac{d}{d\alpha} S\{\alpha,X_{\theta}(\omega_i)\} + \frac{d}{d\alpha}S\{\theta,X_\alpha(\omega_i)\}\right]\Big|_{\alpha=\theta}.
	\end{align*}
	The result follows from the Law of Large Numbers, provided that 
	\[\EE_{\omega \sim P}\left[ \frac{d}{d\alpha} S\{\alpha,X_{\theta}(\omega)\} + \frac{d}{d\alpha}S\{\theta,X_\alpha(\omega)\}\right]\Big|_{\alpha=\theta}=0.\]  The expectation of the first term is $-I(\theta)$, by \citet[Theorem 7.2.1]{lehmann2004elements}. For the second term, we compute
 \begingroup
\allowdisplaybreaks
		\begin{align}
		\EE_{\omega \sim P} \frac{d}{d\alpha} S\{\theta, X_\alpha(\omega)\}\Big|_{\alpha=\theta}
		&=\int_\Omega \frac{d}{d\alpha} S\{\theta,X_\alpha(\omega)\} \Big|_{\alpha=\theta} \pi(\omega)\ d\omega\\
		 &=  \frac{d}{d\alpha}\int_\Omega S\{\theta,X_\alpha(\omega)\} \ \pi(\omega)\ d\omega\Big|_{\alpha=\theta}\label{eq:1}\\
		 &= \frac{d}{d\alpha}\int_{\RR^d}  S(\theta,x) f_\alpha(x)\ dx\Big|_{\alpha=\theta}\label{eq:2}\\
		&= \int_{\RR^d} \frac{d}{d\alpha}S(\theta,x)  f_\alpha(x)\Big|_{\alpha=\theta}\ dx\label{eq:3}\\
		&= \int_{\RR^d} S(\theta,x)\left\{ \frac{d}{d\alpha} f_\alpha(x)\Big|_{\alpha=\theta}\right\}^\top \ dx\\
		&= \int_{\RR^d} S(\theta,x) \left\{\frac{\frac{d}{d\theta} f_\theta(x)}{f_\theta(x)}\right\}^\top f_\theta(x) \ dx\\
		&=\int_{\RR^d} S(\theta,x)S^\top(\theta,x) f_\theta(x) \ dx \\
		&=\EE_{X\sim \theta} \left\{S(\theta,X)S^\top(\theta,X)\right\}\\
		&=I(\theta),
		\end{align}
  \endgroup
		 where for \eqref{eq:1} we use the boundedness of $\Omega$ from (R0) and (R4) to interchange the derivative and integral; for \eqref{eq:2}, we apply a change of variables, using the fact that $f_\alpha(\omega)$ is the density for the random variable $X_\alpha(\omega)$; and for \eqref{eq:3}, we use (R2) and the dominated convergence theorem to change the order of the derivative and integral again.
	\end{proof}

\begin{proof}[Proof of Theorem 2]
We expand $\thz$ about $\thx$ using part 3 of Lemma \ref{lem:score}:
\begin{equation}\label{eq:score0}
    \thz=\thx + I^{-1}\{\thx\} n^{-1} \sum_{i=1}^n S\{\thx,X_{\thx}(\omega_i)\} + o_p(n^{-1/2})
    \end{equation}
  
   The score can be expanded about $\thx=\theta_0$:
\begin{align*}
&n^{-1} \sum_{i=1}^n S\{\thx,X_{\thx}(\omega_i)\}\\
&=n^{-1} \sum_{i=1}^n S\{\theta_0,X_{\theta_0}(\omega_i)\} + \left[\frac{d}{d\twid\theta} \ n^{-1} \sum_{i=1}^n S\{\twid\theta,X_{\twid\theta}(\omega_i)\}\right]\{\thx-\theta_0\}\\
&=n^{-1} \sum_{i=1}^n S\{\theta_0,X_{\theta_0}(\omega_i)\} + o_p(1)O_p(n^{-1/2}),
\end{align*}
where $\twid\theta$ is between $\thx$ and $\theta_0$; by Lemma \ref{lem:fisher}, we justify that the derivative is $o_p(1)$. 

Combining this derivation along with the fact that $I^{-1}(\thx)= I^{-1}(\theta_0)+o_p(1)$ by the continuous mapping theorem, we have the following equation:

  \begin{equation}\label{eq:score1}
    \thz =\thx + I^{-1}(\theta_0) n^{-1} \sum_{i=1}^n S\{\theta_0,X_{\theta_0}(\omega_i)\} + o_p(n^{-1/2}).
\end{equation}

Using the same techniques, we do an expansion for $\thy$ about $\theta^*=2\thx-\thz$:
\begin{align}
    \thy &= \theta^* + I^{-1}(\theta^*) n^{-1} \sum_{i=1}^n S\{\theta^*, X_{\theta^*} (\omega_i)\} + o_p(n^{-1/2})\\
    &= \theta^* + I^{-1}(\theta_0) n^{-1} \sum_{i=1}^n S\{\theta_0, X_{\theta_0}(\omega_i)\} + o_p(n^{-1/2})\label{eq:substitute}\\
    &=\theta^* + (\thz-\thx) + o_p(n^{-1/2})\label{eq:combine}\\
    &=\thx+o_p(n^{-1/2}),\label{eq:final}\\
\nonumber
\end{align}
where line \eqref{eq:substitute} is a similar expansion as used for equation \eqref{eq:score0}, in line \eqref{eq:combine} we substituted the expression from \eqref{eq:score1}, and line \eqref{eq:final}  
uses the fact that as $n\rightarrow \infty$, $\theta^*=2\thx-\thz$ with probability tending to one. Indeed, since $2\thx-\thz$ is a consistent estimator of $\theta_0$, we have that as $n\rightarrow \infty$, $ P(2\thx-\thz\in \Theta)\geq P\{2\thx-\thz \in B(\theta_0)\}\rightarrow 1$.
\end{proof}

\begin{proof}[Proof of Lemma 1]
For a fixed $\theta \in \Theta$, for $\omega\sim P$, the random variable $Y = X_{\theta}(\omega)$ is distributed with probability measure $P X_{\theta}^{-1}$: for any measurable set $E$, $P(Y\in E) = P X_{\theta}^{-1}(E)$. We denote by $P^n_\Omega$ the joint probability measure on $\Omega^n$, and $(P X_{\theta}^{-1})^n$ the joint probability measure on $\RR^{d\times n}$.

Given $\theta^* \in \Theta$, our goal is to derive the probability distribution of the random variables $X_{\theta^*}(\omega_1),\ldots,$ $X_{\theta^*}(\omega_n)$ conditioned on 
the event that $\{\omega_1,\ldots,\omega_n\mid \hat\theta\{X_{\theta^*}(\omega_i)\}= \hat\vartheta\}$. 
However, this event may have zero probability.  Instead, we will condition on $S_{\hat\vartheta,\theta^*}^\delta  = \{\omega_1,\ldots, \omega_n\mid \hat\vartheta\{X_{\theta^*}(\omega)\} \in B_\delta(\hat\vartheta)\}$, where $B_\delta(\hat\vartheta) = \{\theta  \mid \lVert \hat\vartheta - \theta\rVert\leq \delta\}$, which has positive probability. At the end, we will take the limit as $\delta\rightarrow 0$ to derive the desired distribution.
 
 Let $E\subset \RR^{d\times n}$
 be a measurable set. Then 
 \begin{align*}
     &P\{X_{\theta^*}(\omega_1),\ldots, X_{\theta^*}(\omega_n) \in E \mid \omega_1,\ldots, \omega_n \in S_{\theta^*,\hat\vartheta}^\delta\}\\
     &=P(\omega_1,\ldots, \omega_n \in X^{-1}_{\theta^*} E\mid \omega_1,\ldots, \omega_n \in S^{\delta}_{\theta^*,\hat\vartheta})\\
     &= \frac{P^n( X^{-1}_{\theta^*}E \cap S_{\theta^*,\hat\vartheta}^\delta)}{P^n (S_{\theta^*,\hat\vartheta}^\delta)}\\
     &= \frac{(P X_{\theta^*}^{-1})^n (E\cap X_{\theta^*} S^{\delta}_{\theta^*,\hat\vartheta})}{(P X_{\theta^*}^{-1})^n(X_{\theta^*} S_{\theta^*,\hat\vartheta}^\delta)},
 \end{align*}
 where we used the definition of conditional probability and the fact that $X_{\theta^*}^{-1} X_{\theta^*} S_{\theta^*,\hat\vartheta}^\delta = S_{\theta^*,\hat\vartheta}^\delta$.
 
 This last expression shows that $X_{\theta^*}(\omega_1),\ldots, X_{\theta^*}(\omega_n)$ conditioned on $\ul \omega \in S_{\theta^*,\hat\vartheta}^\delta$ is distributed as $f^n_{\theta^*}\{y_1,\ldots, y_n \mid \hat\theta(\ul y) \in B_\delta(\hat\vartheta)\}$. 
  This derivation is valid for all $\delta>0$. Taking the limit as $\delta\rightarrow 0$ gives the desired formula:
  \[Y^{\theta^*}_1,\ldots, Y^{\theta^*}_{n} \Big | \hat\theta(\ul Y^{\theta^*})=\hat\theta(\ul X) \sim f_{\theta^*}^n\{y_1,\ldots, y_n \mid \hat\theta(\ul y) = \hat\theta(\ul X)\}.\]
\end{proof}

  \begin{proof}  [Proof of Theorem 3]
  While the distributions $g$ depend on $n$, we will suppress this dependence for notational simplicity. We can then express the desired KL divergence as follows: 

    First, by the data processing inequality, we can add in the random variable $\hat\theta(\ul X) = \hat\theta(\ul Y)$ to get an upper bound on the KL divergence. We then have closed formulas for the joint distributions $\{X_1,\ldots,X_n,\hat\theta(\ul X)\}$ and $\{Y_1,\ldots, Y_n, \hat\theta(\ul X)\}$.
    \begingroup
\allowdisplaybreaks
  \begin{align}
  \mathrm{KL}&\left(X_1,\ldots,X_n\middle|\middle| Y_1,\ldots, Y_n\right)\\
  &\leq 
  \mathrm{KL}\left\{X_1,\ldots,X_n,\hat\theta(\ul X)\middle|\middle| Y_1,\ldots, Y_n,\hat\theta(\ul X)\right\}\\
  &=\mathrm{KL}\left[f_\theta^n\{\ul x\mid \hat\theta(\ul x)\}g_{\theta}\{\hat\theta(x)\} \middle|\middle| f_{\theta_n}^n\{\ul x\mid \hat\theta(\ul x)\}g_{\theta}\{\hat\theta(\ul x)\}\right]\\
    &=\EE_{\hat\vartheta \sim g(\cdot\mid \ul X)}\EE_{\ul X\sim f_\theta} \log \left\{ \frac{f_\theta(\ul X\mid\hat\vartheta) g_\theta(\hat\vartheta) }
    {f_{\theta_n}(\ul X\mid \hat\vartheta) g_\theta(\hat\vartheta) } \right\}\label{eq:kl2}\\
    &=\EE_{\hat\vartheta \sim g(\cdot\mid \ul X)}\EE_{\ul X\sim f_\theta} \log \left\{ \frac{f_\theta (\ul X)g(\hat\vartheta\mid \ul X) }
    {f_{\theta_n}(\ul X\mid \hat\vartheta) g_\theta(\hat\vartheta)}\right\}\label{eq:kl3}\\
    &=\EE_{\hat\vartheta \sim g(\cdot\mid \ul X)}\EE_{\ul X\sim f_\theta} \log \left[\frac{f_\theta(\ul X) g(\hat\vartheta\mid \ul X)}{\{f_{\theta_n}(\ul X) g(\hat\vartheta\mid \ul X)/g_{\theta_n}(\hat\vartheta)\}  g_\theta(\hat\vartheta)}\right]\label{eq:kl4}\\
         &=\EE_{\hat\vartheta \sim g(\cdot\mid \ul X)}\EE_{\ul X\sim f_\theta} \log \left\{\frac{f_{\theta}(\ul X) {g(\hat\vartheta\mid \ul X)}}{f_{\theta_n}(\ul X){g(\hat\vartheta\mid \ul X)}}\right\}\\
      &\phantom{=}+ \EE_{\hat\vartheta \sim g(\cdot\mid \ul X)}\EE_{\ul X \sim f_{\theta}} \log \left\{\frac{g_{\theta_n}(\hat\vartheta)}{g_{\theta}(\hat\vartheta)}\right\}\label{eq:kl5}\\
      &=-\EE_{\ul X\sim f_\theta} \log \left\{\frac{f_{\theta_n}(\ul X) }{f_{\theta}(\ul X)}\right\} + \EE_{\hat\vartheta \sim g_\theta} \log \left\{\frac{g_{\theta_n}(\hat\vartheta)}{g_{\theta}(\hat\vartheta)}\right\}\label{eq:kl6},
  \end{align}
  \endgroup
 
 where line \eqref{eq:kl2} simply applies the definition of KL divergence, and line \eqref{eq:kl4} uses the definition of conditional distribution.
  
At this point, we need to compute the two expectations of line \eqref{eq:kl6}, and show that everything cancels except for an $o_p(1)$ term. 

We write $\ell(\theta\mid \ul x) = \sum_{i=1}^n \log f_\theta(x_i)$. Using our assumptions, we can expand $\ell(\theta_n\mid \ul x)$:
\begin{align*}
    \ell(\theta_n\mid \ul x) &= \ell(\theta\mid \ul x) + (\theta_n - \theta)^\top \nabla \ell(\theta\mid \ul x) + \frac 12 (\theta_n - \theta)^\top \nabla^2 \ell(\theta\mid \ul x) (\theta_n- \theta)\\
    &\phantom{=}+\frac{1}{6} \xi^*\sum_{i,j,k} (\theta_n-\theta)_i(\theta_n-\theta)_j(\theta_n-\theta)_k \sum_{s=1}^ng_{ijk}(x_s),
    \end{align*}
    where $|\xi^*|\leq1$ and $g_{ijk}(x)$ is an upper bound for $\left|\frac{\partial^3 \ell(\theta\mid \ul x)}{\partial \theta_i\theta_j\theta_k}\right|$ for a ball about $\theta$, which exists by (R3). These expansions are based on those from \citet{serfling2009approximation}. 
Applying $\EE_{\ul X \sim f_\theta}$ to this derivation gives 
\begin{equation}\label{eq:klF}
\begin{split}
    \EE_{\ul X\sim f_\theta} \log \left\{ \frac{f_{\theta_n}(\ul X)}{f_\theta(\ul X)}\right\}
    &= 0-\frac {n}{2} (\theta_n - \theta)^\top I(\theta)(\theta_n - \theta) \\
    &\phantom{=}+ O(1)\frac{n}{6} \sum_{i,j,k} \{\EE g_{i,j,k}(x)\} (\theta_n-\theta)_i(\theta_n- \theta)_j (\theta_n - \theta)_k,\\
    &= \frac{-n}{2}  (\theta_n - \theta)^\top I(\theta)(\theta_n - \theta) + O(n) \lVert \theta_n-\theta\rVert^3
\end{split}
\end{equation}
where the first term is zero as the expected value of the score function is zero by (R3),
the second term uses \citet[Theorem 7.2.1]{lehmann2004elements} and (R3). The $O(1)$ factor in the third term is based on the fact that $|\xi^*|\leq 1$. Finally, note that $\sum_{i,j,k}\{\EE g_{i,j,k}(x)\}(\theta_n-\theta)_i(\theta_n-\theta)_j(\theta_n-\theta)_k\leq p^3 \sup_{i,j,k} \{\EE g_{i,j,k}(x)\} \lVert \theta_n-\theta\rVert^3_\infty = O(1) \lVert \theta_n-\theta\rVert^3$. Note that all norms are equivalent in $\RR^p$, so they can be interchanged up to a factor of $O(1)$.

Next, we will derive a similar formula for $\log g_{\theta^*}(\hat\vartheta)$:
\begin{equation}
\begin{split}
    \log g_{\theta_n}(\hat\vartheta)
    &=\log g_{\theta}(\hat\vartheta)
    +\nabla \log g_\theta(\hat\vartheta)(\theta_n - \theta)\\
    &\phantom{=}+\frac12 (\theta_n- \theta)^\top \nabla^2 \log g_\theta(\hat\vartheta) (\theta_n - \theta)\\
    &\phantom{=}+\frac n6\xi_2^* \sum_{i,j,k} (\theta_n-\theta)_i(\theta_n-\theta)_j(\theta_n-\theta)_k  G_{i,j,k}(\hat\vartheta),
    \end{split}
    \label{eq:expansionG}
\end{equation}
where $|\xi_2^*|\leq 1$. In order to apply the expectation $\EE_{\hat\vartheta \sim \theta}$ to this equation, we will first show $\EE_{\hat\vartheta\sim \theta} \nabla \log g_\theta(\hat\vartheta)=0$ and $\EE_{\hat\vartheta\sim \theta} \nabla^2 \log g_\theta(\hat\vartheta)=-nI(\theta)+o(n)$.

\begin{align*}
    \left\{\EE_{\hat\theta \sim \theta} \nabla \log g_\theta(\hat\vartheta)\right\}_j
    &= \int \left\{\frac{\partial}{\partial \theta_j} \log g_\theta(\hat\vartheta)\right\} g_\theta(\hat\vartheta)\ d\hat\vartheta\\
    &= \int \frac{\partial}{\partial \theta_j} g_\theta(\hat\vartheta)\ d\hat\vartheta\\
    &=\int \frac{\partial}{\partial \theta_j} \int_x f_\theta(x) g(\hat\vartheta\mid x) \ dx\ d\hat\vartheta\\
    &= \int_{\hat\theta} \int_x \frac{\partial}{\partial \theta_j} f_\theta(x)  g(\hat\vartheta\mid x) \ dx\ d\hat\vartheta\\
    &=\frac{\partial}{\partial \theta_j} \int\int f_\theta(x) g(\hat\vartheta\mid x) \ dx\ d\hat\vartheta\\
    &=0,
\end{align*}
where we use the assumption (R3) that $\left|\frac{\partial}{\partial \theta_j}f_\theta(x)\right|$ is  bounded above by an integrable function,  (R5) that $g(\hat\vartheta\mid x)$ is bounded, and the dominated convergence theorem to interchange the derivative and the integral. 

Next we work on the second derivative:

\begin{align*}
    &\left(\EE_{\hat\vartheta \sim \theta} \nabla^2 \log g_\theta(\hat\vartheta)\right)_{j,k}\\
    &= \int \left\{\frac{\partial^2}{\partial \theta_j\partial \theta_k} \log g_\theta(\hat\vartheta)\right\} g_\theta(\hat\vartheta) \ d\hat\vartheta\\
    &=\int \frac{g_\theta(\hat\vartheta) \frac{\partial^2}{\partial \theta_j\partial\theta_k} g_\theta(\hat\vartheta) - \frac{\partial}{\partial \theta_j} g_\theta(\hat\vartheta) \left\{\frac{\partial}{\partial \theta_k} g_\theta(\hat\vartheta)\right\}^\top }{g_\theta^2(\hat\vartheta)} g_\theta(\hat\vartheta) \ d\hat\vartheta\\
    &=0-\EE_{\hat\vartheta\sim \theta}\left( \nabla \log g_\theta(\hat\vartheta) \nabla^\top \log g_\theta(\hat\vartheta)\right)_{j,k},\\
\end{align*}
where we used (R3) along with the dominated convergence theorem to set the first term equal to zero. We see that $\EE \nabla^2 \log g_\theta(\hat\vartheta) = -I_{\thx}(\theta)$, where $I_{\thx}$ represents the Fisher information of the random variable $\hat\theta(\ul X)\sim g$. It is our current goal to show that $I_{\thx}(\theta) = nI(\theta)+o(n)$, where $I(\theta)$ is the Fisher information for one sample $X\sim f_\theta$. First note that by the data processing inequality \citep{zamir1998proof}, $I_{\hat\theta(\ul X)}(\theta)\leq I_{X_1,\ldots, X_n}(\theta) = nI(\theta)$, where the inequality represents the positive-definite ordering of matrices. Next, we need to find a matching lower bound. By the Cram\'er Rao lower bound, we have that 
\[ \{I_{\hat\theta(\ul X)}(\theta)\}^{-1} \leq \var\{\hat\theta(\ul X)\}+o(1/n),\]
where $\var\{\thx\}$ is the covariance matrix of the random variable $\thx$, and we used the fact that $\thx$ is asymptotically unbiased. By the efficiency of $\hat\theta(\ul X)$, we have that 
\[\var\{\hat\theta(\ul X)\} = n^{-1} I^{-1}(\theta) + o(1/n).\]
We then have 
\begin{align*}
    I_{\hat\theta(\ul X)}(\theta) &\geq \left\{n^{-1} I^{-1}(\theta) + o(1/n)\right\}^{-1}\\
    &= n\{I^{-1}(\theta) + o(1)\}^{-1}\\
    &=n \{I(\theta) + o(1)\},
\end{align*}
where for the last equality, we use the following matrix identity:
\[(A+B)^{-1} = A^{-1} - A^{-1} B(A+B)^{-1},\]
where we set $A = I^{-1}(\theta)$ and $B=o(1)$.

Combining our results, we have that 

\[\EE_{\hat\vartheta \sim \theta} \nabla^2 \log g_\theta(\hat\vartheta) = -I_{\hat\theta(\ul X)}(\theta) = n\{-I(\theta) + o(1)\}.\]
Finally, applying the expectation to equation \eqref{eq:expansionG}, we have 
\begin{equation}\label{eq:g}
\begin{split}
\EE_{\hat\vartheta \sim \theta}\log \left\{\frac{g_{\theta_n}(\hat\vartheta)}{g_\theta(\hat\vartheta)}\right\}
&=0 - \frac{n}{2} (\theta_n - \theta) \{I(\theta) + o(1)\}(\theta_n-\theta)\\
&\phantom{=}+ O(1) \frac n6 \sum_{i,j,k} \{\EE G_{i,j,k}(\hat\vartheta)\} (\theta_n-\theta)_i(\theta_n-\theta)_j(\theta_n-\theta)_k\\
&=\frac{-n}{2}(\theta_n-\theta)^\top I(\theta)(\theta_n-\theta)+ o(n)\lVert \theta_n-\theta\rVert^2\\
&\phantom{=}+ O(n) \lVert \theta_n-\theta\rVert^3.
\end{split}
\end{equation}
Combining equations \eqref{eq:klF} and \eqref{eq:g}, we have 
\begin{align*}
\mathrm{KL}&\left(X_1,\ldots,X_n\middle|\middle| Y_1,\ldots, Y_n\right)\\
&\leq\mathrm{KL}\left[f_\theta^n\{\ul x\mid \hat\theta(\ul x)\}g_{\theta}\{\hat\theta(x)\} \middle|\middle| f_{\theta_n}^n\{\ul x\mid \hat\theta(\ul x)\}g_{\theta}\{\hat\theta(\ul x)\}\right]\\
    &=-\EE_{\ul X\sim f_\theta} \log \left\{\frac{f_{\theta_n}(\ul X) }{f_{\theta}(\ul X)}\right\}
      + \EE_{\hat\vartheta \sim g_\theta(\hat\vartheta)} \log \left\{\frac{g_{\theta_n}(\hat\vartheta)}{g_{\theta}(\hat\vartheta)}\right\}\\
      &=o(n)\lVert \theta_n-\theta\rVert^2 + O(n) \lVert \theta_n-\theta\rVert^3.
\end{align*}

\end{proof}
  
  \section{Additional Simulation Results}
  				\begin{figure}[ht]
		\includegraphics[width=.48\textwidth]{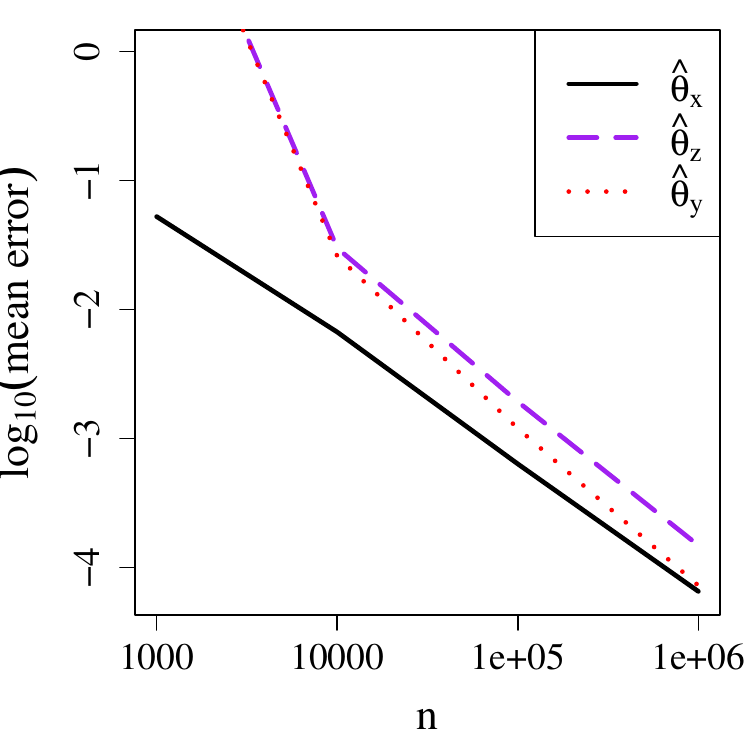}
			\includegraphics[width=.48\textwidth]{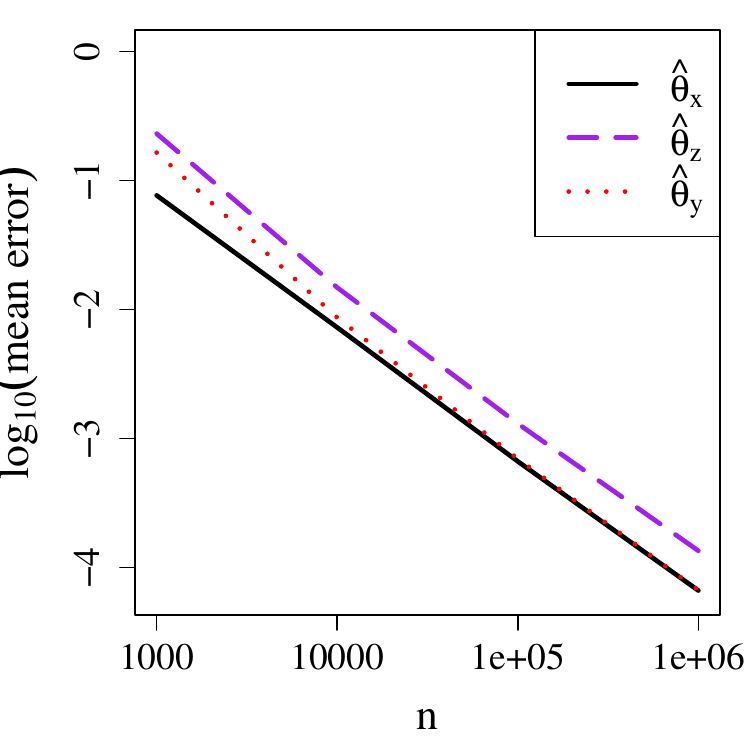}
		\includegraphics[width=.48\textwidth]{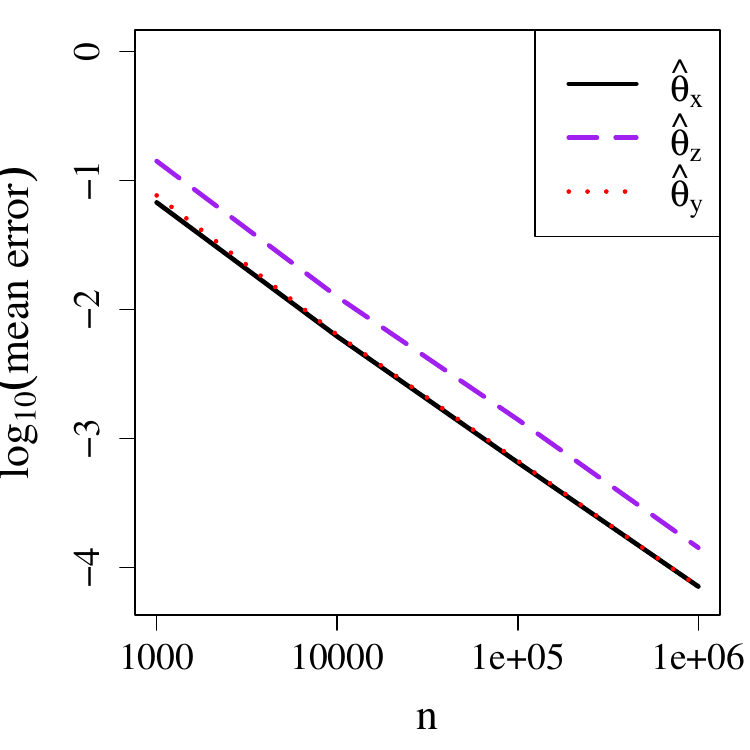}
			\includegraphics[width=.48\textwidth]{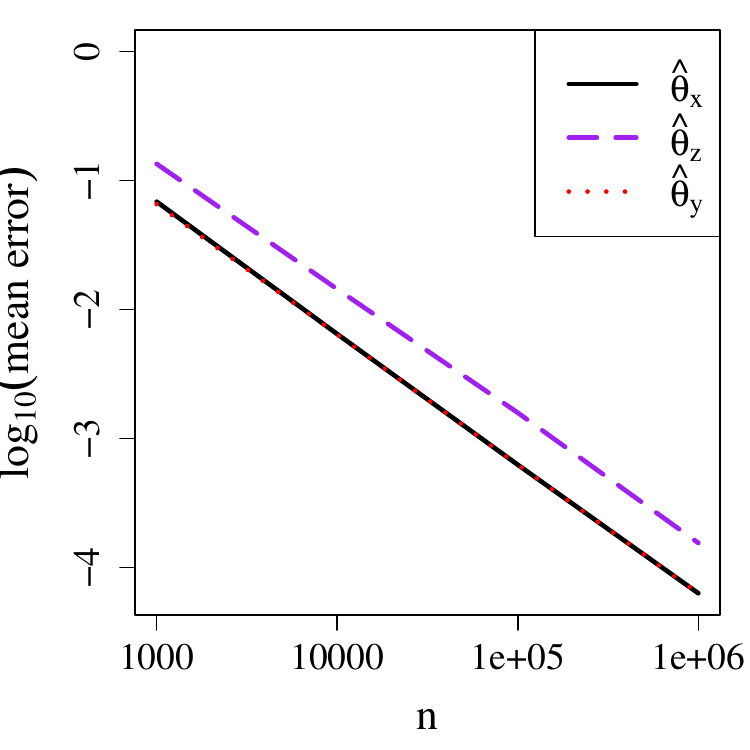}
			\caption{Additional simulations for Section 6.3. In normal reading order, $\ep=.5,2,4,\infty$. Note that Figure 1(b) in the main paper is for $\ep=1$}
			\label{fig:beta2}
		\end{figure}
\subsection{Differentially private beta synthetic data}
  In this section, we consider additional values of $\epsilon$ that are used in the Section 6.3 experiment on differentially private beta distributed synthetic data. All other simulation parameters are identical. We varied $\ep=.5,2,4,\infty$ (note that $\ep=1$ appears in Figure 1(b) in the main paper). In Figure \ref{fig:beta2}, we see that at all values of $\ep$, $\thy$ is very close to $\hat\theta_{DP}$. However, with smaller $\ep$ it requires a larger sample size before the performance of $\thy\approx \hat\theta_{DP}$ is similar to the MLE $\thx$. Note that even with $\ep=\infty$, the performance of $\thz$ does not approach that of $\thx$. 

  \subsection{DP two sample proportion test}
  In this section, we repeat the experiment of Section 6.4 with varying values of $\epsilon$. All other simulation parameters are the same. We varied $\ep=.5,2,4,10$ (note that $\ep=1$ appears in Figure 1(b) in the main paper). In Figure \ref{fig:prop_add1}, we plot the $p$-values of both the one-step and parametric bootstrap tests. We see that the parametric bootstrap $p$-values are very conservative for all values of $\epsilon$, whereas the one-step $p$-values are fairly well-calibrated, although sometimes slightly inflated. In Figure \ref{fig:prop_add2} we plot the power of the two tests for the different $\epsilon$ values. We see that the relative performance of one-step compared to parametric bootstrap is unchanged as we vary $\epsilon$.

		\begin{figure}[ht]
		\includegraphics[width=.48\textwidth]{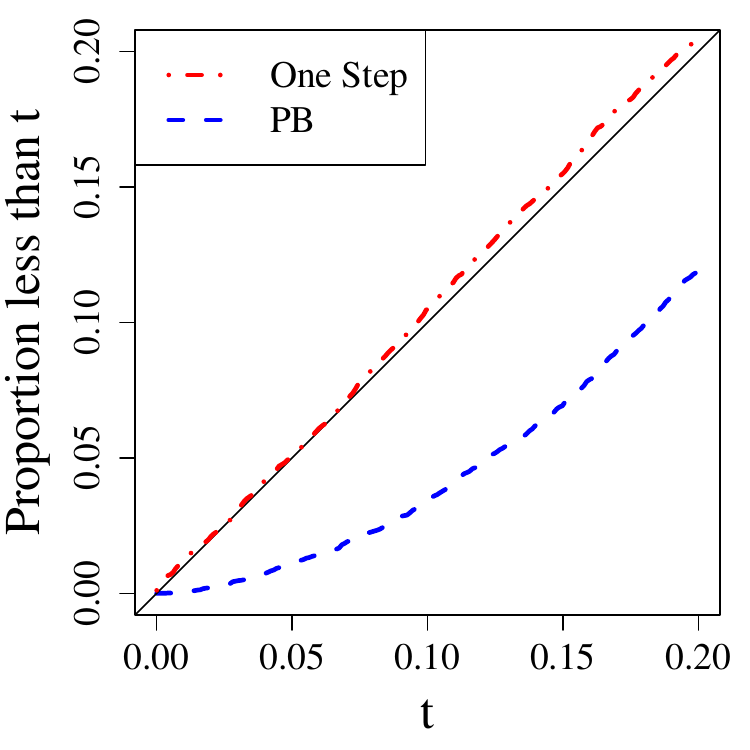}
			\includegraphics[width=.48\textwidth]{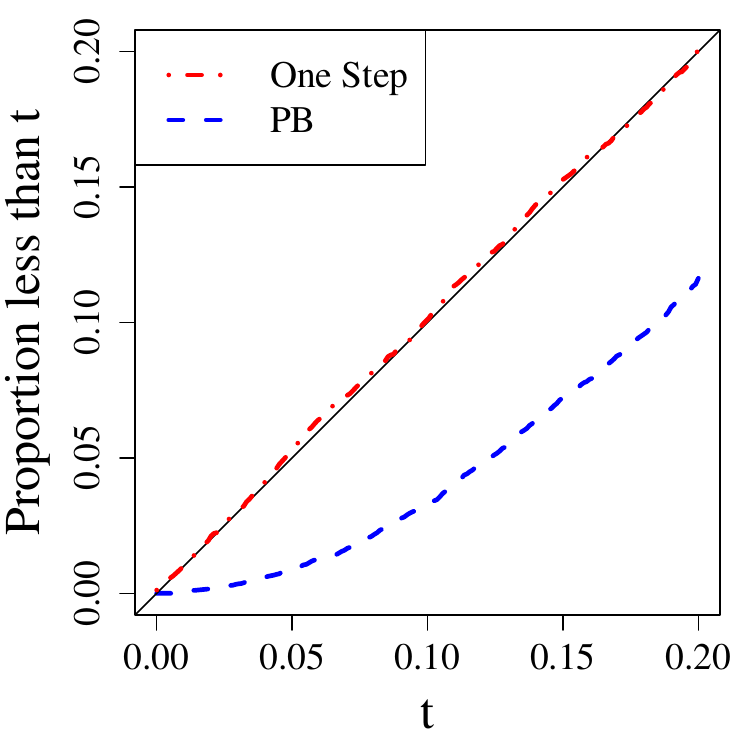}
		\includegraphics[width=.48\textwidth]{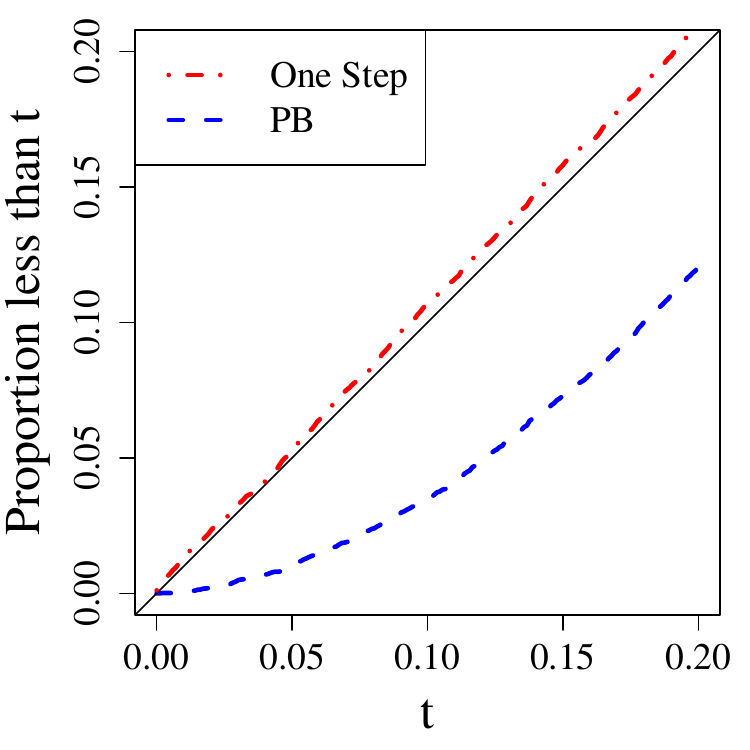}
			\includegraphics[width=.48\textwidth]{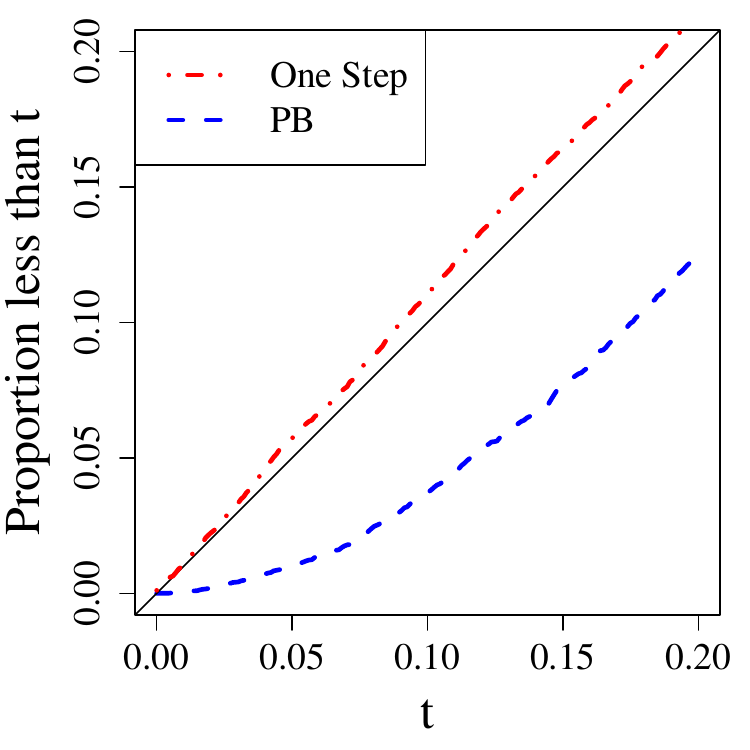}
			\caption{Additional simulations for Section 6.4. In normal reading order, $\ep=.5,2,4,10$. Note that Figure 2(a) in the main paper is for $\ep=1$}
			\label{fig:prop_add1}
		\end{figure}

		\begin{figure}[ht]
		\includegraphics[width=.48\textwidth]{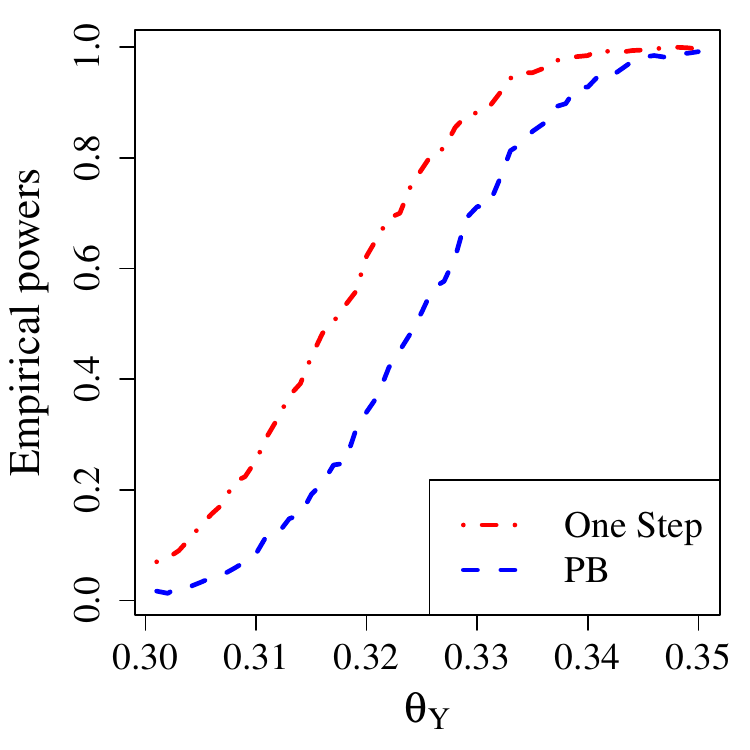}
			\includegraphics[width=.48\textwidth]{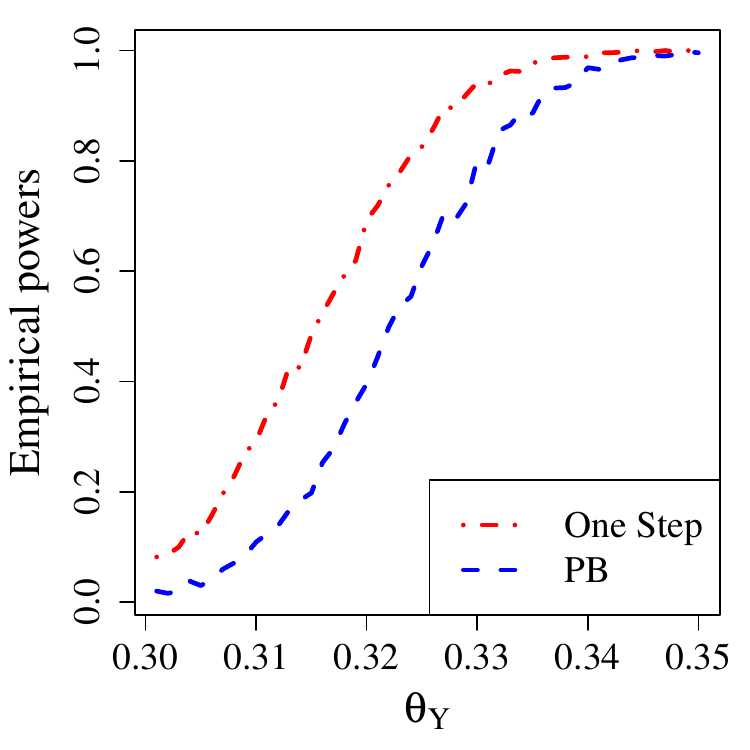}
		\includegraphics[width=.48\textwidth]{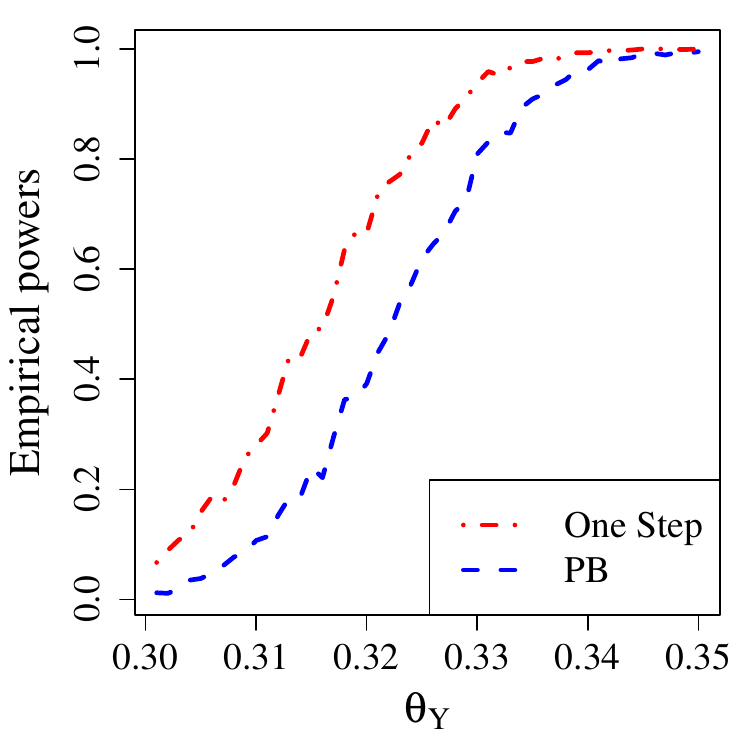}
			\includegraphics[width=.48\textwidth]{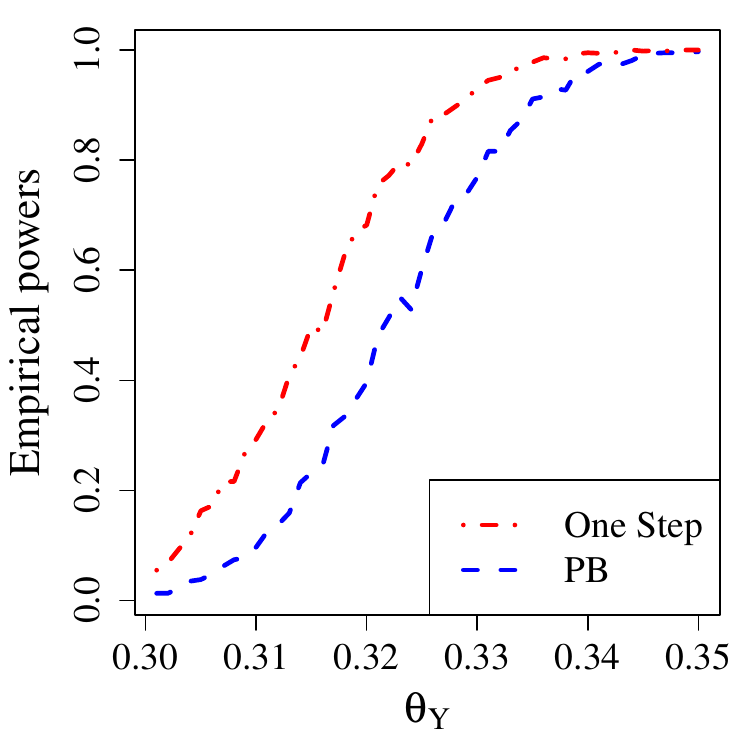}
			\caption{Additional simulations for Section 6.4. In normal reading order, $\ep=.5,2,4,10$. Note that Figure 2(b) in the main paper is for $\ep=1$}
			\label{fig:prop_add2}
		\end{figure}

\end{document}